\documentclass[11pt,reqno]{amsart}
\usepackage{mathrsfs}
\usepackage{amssymb}
\usepackage{amsfonts}
\usepackage{amsmath}
\usepackage{txfonts}
\usepackage{dsfont}
\usepackage{amscd}
\usepackage{tabularx}
\usepackage{booktabs}
\usepackage{float}
\usepackage{graphics}
\usepackage{tikz}
\usetikzlibrary{decorations.text}
\usepackage[colorlinks, linkcolor=blue, anchorcolor=blue, citecolor=blue]{hyperref}
\begin{document}
\setlength{\oddsidemargin}{0cm} \setlength{\evensidemargin}{0cm}
\baselineskip=20pt

\theoremstyle{plain} \makeatletter
\newtheorem{theorem}{Theorem}[section]
\newtheorem{Proposition}[theorem]{Proposition}
\newtheorem{Lemma}[theorem]{Lemma}
\newtheorem{Corollary}[theorem]{Corollary}
\newtheorem{Question}[theorem]{Question}

\theoremstyle{definition}
\newtheorem{notation}[theorem]{Notation}
\newtheorem{exam}[theorem]{Example}
\newtheorem{proposition}[theorem]{Proposition}
\newtheorem{conj}{Conjecture}
\newtheorem{prob}[theorem]{Problem}
\newtheorem{remark}[theorem]{Remark}
\newtheorem{claim}{Claim}
\newtheorem{Definition}[theorem]{Definition}

\newcommand{\SO}{{\mathrm S}{\mathrm O}}
\newcommand{\SU}{{\mathrm S}{\mathrm U}}
\newcommand{\Sp}{{\mathrm S}{\mathrm p}}
\newcommand{\so}{{\mathfrak s}{\mathfrak o}}
\newcommand{\Ad}{{\mathrm A}{\mathrm d}}
\newcommand{\m}{{\mathfrak m}}
\newcommand{\g}{{\mathfrak g}}
\newcommand{\h}{{\mathfrak h}}


\numberwithin{equation}{section}
\title[Curvatures of metric Jordan algebras]{Curvatures of metric Jordan algebras}
\author{Hui Zhang}
\address [Hui Zhang]{School of Mathematics, Southeast University, Nanjing 210096, P.R. China}\email{huizhang@mail.nankai.edu.cn}
\author{Zaili Yan}
\address[Zaili Yan]{ School of Mathematics and Statistics, Ningbo University, Ningbo, Zhejiang Province, 315211,  People's Republic of China}\email{yanzaili@nbu.edu.cn}
\author{Zhiqi Chen$^{*}$}
\address[Zhiqi Chen]{School of Mathematics and Statistics, Guangdong University of Technology, Guangzhou 510520, P.R. China}\email{chenzhiqi@nankai.edu.cn}
\thanks{* Corresponding author}

\begin{abstract}
In this paper, we study  metric Jordan algebras,  i.e.,  Jordan algebras  with an inner product, and draw parallels to metric Lie algebras   studied   extensively for  understanding left-invariant Riemannian metrics on Lie groups.  We first define the Jordan-Levi-Civita connection on metric Jordan algebras which share  analogous construction   of invariant super-connections on Lie super-groups, then show its uniqueness. Utilizing this connection,  we introduce three natural curvature tensors on metric Jordan algebras, and obtain  the corresponding   formulas. Based on the   curvature formulas, we prove that every formally real Jordan algebra admits   a metric of non-positive Jordan curvature and a Jordan-Einstein metric of negative Jordan  scalar curvature. Moreover, for nilpotent Jordan algebras, we prove that they  admit no Jordan-Einstein metrics.
\end{abstract}
\maketitle

\section{Introduction}

A \textit{left symmetric algebra} (write LSA for short)  is   a vector space $V$  with a  binary operation $(x,y)\rightarrow xy$  satisfying
\begin{align}\label{+}
(xy)z-x(yz)=(yx)z-y(xz),~~\forall x,y,z\in V.
\end{align}
Define the associator as $(x,y,z):=(xy)z-x(yz),~\forall x,y,z\in V.$
Then equation (\ref{+}) is exactly the following identity: $$(x,y,z)=(y,x,z),~\forall x,y,z\in V.$$
That is,  the associator  is symmetric in the left two variables. It is well-known that the commutator of a  left symmetric  algebra defines  a natural Lie algebra structure, i.e.,
\begin{align*}
[x,y]:=xy-yx, ~~\forall x,y\in V.
\end{align*}
For this reason, left symmetric algebras are also referred to as \textit{pre-Lie algebras} in the literature. As emphasized  in \cite{B2021,Burde2006,Vinberg}, left symmetric algebras hold  significance in both   geometry and physics. Among the important results, we mention the following theorem

\begin{theorem}\label{CL}
For a Lie group  $G$  with Lie algebra $\g,$  there is a one-to-one correspondence between   left-invariant  flat torsion-free affine connections  on $G$ \textnormal{(}or just $\g$\textnormal{)} and  LSA-structures on $\g$.
\end{theorem}

 Modifying slightly the equation (\ref{+}) by a sign, i.e.,
\begin{align}\label{-}
(xy)z-x(yz)=-(yx)z+y(xz),~~\forall x,y,z\in V,
\end{align}
or equivalently,  $(x,y,z)=-(y,x,z),~\forall x,y,z\in V,$  it turns out that the new  operation
\begin{align*}
x\circ y:=xy+yx, ~~\forall x,y\in V
\end{align*}
 defines on $V$ a Jordan algebra structure (see Thm.~\ref{ZFDCDS}).
This  is  expected, see  the following diagram.
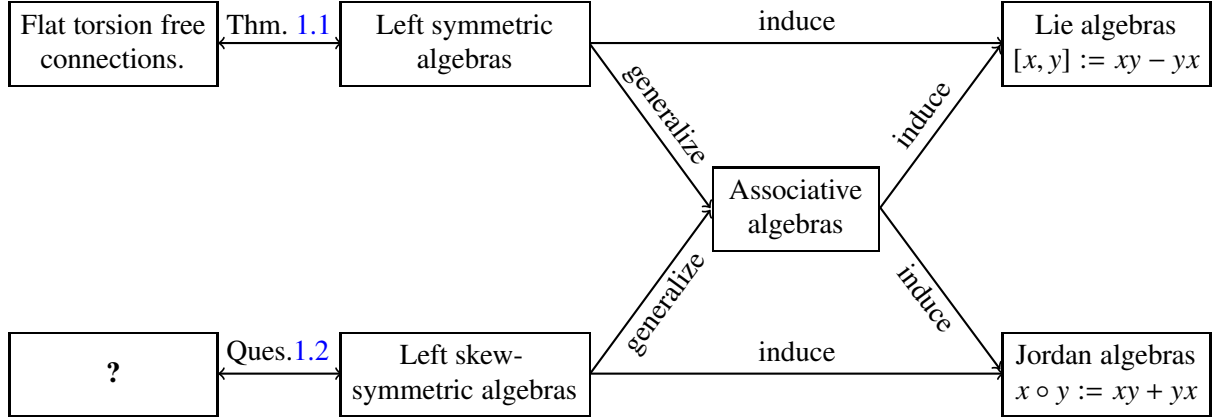
\begin{center}\setlength{\unitlength}{1.1mm}
\begin{picture}(145,55)
\put(0,40){\framebox(25,10){$\begin{array}{cc}\text{Flat torsion free}\\ \text{connections.}\end{array}$}}
\put(25,45){\begin{tikzpicture}
 \draw[<->,thick, black] (0,0) to (1.65, 0);
\end{tikzpicture}}
\put(25,45){\begin{tikzpicture}
 \filldraw[black]  node[anchor=west]{Thm.~\ref{CL}};
\end{tikzpicture}}
\put(40,40){\framebox(30,10){$\begin{array}{cc}\text{Left symmetric}\\ \text{algebras}\end{array}$}}
\put(70,25){\begin{tikzpicture}
 \draw[->,thick, black, postaction={decorate,decoration={raise=1ex,text along path,text align=center,{ text={ generalize}}}}] (0,0) to (1.6, -2.18);
\end{tikzpicture}}
\put(70,45){\begin{tikzpicture}
 \draw[->,thick, black,postaction={decorate,decoration={raise=1ex,text along path,text align=center,{ text={ induce}}}}] (0,0) to (5.5, 0);
\end{tikzpicture}}
\put(0,0){\framebox(25,10){$\begin{array}{cc}\textbf{?}\end{array}$}}
\put(25,5){\begin{tikzpicture}
 \draw[<->,thick, black] (0,0) to (1.65, 0);
\end{tikzpicture}}
\put(25,5){\begin{tikzpicture}
 \filldraw[black]  node[anchor=west]{Ques.\ref{JT}};
\end{tikzpicture}}
\put(70,5){\begin{tikzpicture}
 \draw[->,thick, black,postaction={decorate,decoration={raise=-2ex,text along path,text align=center,{ text={ generalize}}}}] (0,0) to (1.6, 2.18);
\end{tikzpicture}}
\put(40,0){\framebox(30,10){$\begin{array}{cc}\text{Left skew- }\\ \text{symmetric algebras}\end{array}$}}
\put(70,5){\begin{tikzpicture}
 \draw[->,thick, black, postaction={decorate,decoration={raise=1ex,text along path,text align=center,{ text={ induce}}}}] (0,0) to (5.5, 0);
\end{tikzpicture}}
\put(85,20){\framebox(20,10){$\begin{array}{cc}\text{Associative}\\ \text{algebras} \end{array}$}}
\put(105,25){\begin{tikzpicture}
 \draw[->,thick, black,postaction={decorate,decoration={raise=1ex,text along path,text align=center,{ text={ induce}}}}] (0,0) to (1.6, 2.16);
\end{tikzpicture}}
\put(105,5.5){\begin{tikzpicture}
 \draw[->,thick, black,postaction={decorate,decoration={raise=-2ex,text along path,text align=center,{ text={ induce}}}}] (0,0) to (1.6, -2.16);
\end{tikzpicture}}
\put(120,40){\framebox(25,10){$\begin{array}{cc}\text{Lie algebras}\\ {[x,y]:=xy-yx}\end{array}$}}
\put(120,0){\framebox(25,10){$\begin{array}{cc}\text{Jordan algebras}\\ x\circ y:=xy+yx\end{array}$}}
\end{picture}
\end{center}
Here, a left skew-symmetric algebra  means an algebra that satisfies the equation (\ref{-}), see Def.~\ref{lssa}. Naturally, one may ask
\begin{Question}\label{JT}
Is   there  an  analogy of Theorem~\textnormal{\ref{CL}} for Jordan algebras?
\end{Question}


On the other hand, if $(G,\langle\cdot,\cdot\rangle)$ be a connected Lie group with a left-invariant Riemannian metric,  then we may  identify $(G,\langle\cdot,\cdot\rangle)$ with the metric Lie algebra $(\g,\langle\cdot,\cdot\rangle)$.
The Ricci  curvature of $(G,\langle\cdot,\cdot\rangle)$   can be simply expressed  in terms of  the Lie brackets  of $\g$ and the inner product $\langle\cdot,\cdot\rangle$. In particular, the Ricci operator of $(\g,\langle\cdot,\cdot\rangle)$ can be written  as follows (see \cite{besse87book,Lauret2010})
\begin{align}\label{Ricciformula0}
\textnormal{Ric}=\textnormal{M}-\frac{1}{2}B-S(\operatorname{ad}H),
\end{align}
where $B$ is the operator defined by  the Killing form of $\g$, and $H$ is the mean curvature vector of $(\g,\langle\cdot,\cdot\rangle)$.
The first term on the right side of equation (\ref{Ricciformula0}) remained mysterious until Lauret, in \cite{Lauret03,Lauret2003}, revealed that the map \textnormal{M} coincides up to a scalar with the \textit{moment map} for the variety of Lie algebras. Using
this observation, he proved that every Einstein solvmanifold is standard
(\cite{Lauret2010}), and later, provided a complete characterization of solvsolitons (\cite{Lauret2011}).
As noted  in \cite{Lauret2020}, the moment map can be naturally extended in other classes of algebras such
as associative algebras, Jordan algebras and etc, which have recently been investigated  in some works, see \cite{GKM,ZY,ZCL}.
In light of  this, one may also ask
\begin{Question}\label{Ques}
{Does  there exist an analogous `Ricci curvature tensor' on  metric Jordan algebras?}
\end{Question}


Recall that an algebra $\mathscr{A}$ is said to be a \textit{Jordan algebra}  if, for all $X,Y$ in $\mathscr{A}$,
\begin{align*}
X\circ Y = Y\circ  X,\quad
X\circ((X\circ X)\circ Y) = (X\circ X)\circ (X\circ Y).
\end{align*}
Jordan algebra was   introduced by P. Jordan in an attempt to generalize the formalism of
quantum mechanics in 1930s (\cite{J1933,JVP1934}), which later has been proved to be very versatile in both mathematics and physics. For instance, the link between Jordan algebras, symmetric spaces and harmonic analysis
\cite{Bertram2000,Chu2008,Chu2012,Chu2021,Chu2022,FK1994,Koecher1999}, the connection between Jordan algebras and quantum theories \cite{Baez2022,FFMP2014}, and the role Jordan algebras  in information geometry \cite{CJS2023}. We also refer to \cite{Ior2011}
for the applications of Jordan algebras in other fields.

This paper focuses on {metric Jordan algebras}, and we shall study  metric Jordan algebras in an analogy of metric Lie algebras;  the latter have been extensively studied in the literature in order to understand left-invariant Riemannian metrics on Lie groups (see  \cite{M76}).
See  TABLE~\ref{LieJor} below for a comparison of  the well-known results in metric Lie algebras with the  ones obtained in this paper for  metric Jordan algebras.
 \begin{table}[H]\label{LieJor}
	\centering
	\begin{tabularx}{\textwidth}{|XlX|}
		\toprule
	 \quad \quad\quad\quad\textbf{Metric Lie algebras} &\textbf{V.S.} &  \quad \quad\quad\quad\textbf{Metric Jordan algebras}  \\
		\midrule
 Every metric Lie algebra admits a unique Levi-Civita connection& & Every metric Jordan algebra admits a unique
Jordan-Levi-Civita connection\\
		\midrule
          Riemann curvature tensor, Ricci curvature \quad tensor,
scalar curvature&& Jordan curvature tensor, Jordan Ricci curvature tensor, Jordan scalar curvature\\
        \midrule
 For an associative inner product on $\g$, the
Levi-Civita connection is: $\nabla_XY=\frac{1}{2}[X,Y]$, $\forall X,Y\in\g$&& For an associative inner product on $\mathscr{A}$, the Jordan-Levi-Civita connection is: $\nabla_XY=\frac{1}{2}XY$, $\forall X,Y\in\mathscr{A}$\\
         \midrule
 Every compact Lie algebra admits a metric of
non-negative Riemann curvature
 && Every formally real Jordan algebra  admits a
metric of non-positive Jordan curvature        \\
          \midrule
 Every compact simple Lie algebra admits a \quad
Einstein metric of positive scalar curvature,  which is given by the Killing form
        && Every simple formally real Jordan algebra
\quad admits a Jordan-Einstein metric of negative
Jordan scalar curvature \\
           \midrule
     A nontrivial nilpotent Lie algebra admits no \quad
 Einstein metrics&&A nontrivial nilpotent Jordan algebra admits no Jordan-Einstein metrics   \\
		\bottomrule
	\end{tabularx}
	\caption{Lie structures v.s. Jordan structures}
\end{table}%
We also note that algebras  with a scalar product (pseudo-inner product, symplectic structure etc),
are of special interests in both geometry and
physics (see \cite{BG2018,BGZ2019,DFMR2009, Drin1983, Fischer2019a,Fischer2019b, KO2006,MP1985}).

This paper is organised as follows: In Sect.~\ref{Basic}, we recall some basic concepts and results of the moment map (in algebras), and the Jordan algebras, respectively.

In Sect.~\ref{Connection}, we first introduce the  concepts of \textit{connection}, \textit{torsion-free connection}, and \textit{Jordan-Levi-Civita connection} on Jordan algebras (Def.~\ref{JLCC}), which share  analogous construction   of invariant super-connections on super-manifolds corresponding to the odd component (see \cite{Go}). Then we show
that every metric Jordan algebra admits a unique Jordan-Levi-Civita connection (Thm.~\ref{unique}), which might be regarded  as a special case of \cite[Thm~1]{Go} within the context of invariant metrics on Lie super-groups. Utilizing the connection,  three natural  quantities   on metric Jordan algebras are introduced,
i.e., \textit{Jordan curvature tensor} (Def.~\ref{JCT}), \textit{Jordan Ricci curvature tensor} (Def.~\ref{JRT}), and \textit{Jordan scalar curvature} (Def.~\ref{JSC}).
We obtain the specific formulas for  Jordan Ricci curvature tensor (Thm.~\ref{RF})  and  Jordan scalar curvature (\ref{scf}), which  provide an appropriate  solution to Question~\ref{Ques}. Besides, we also get a  Jordan version of Theorem~\ref{CL}  (see Cor.~\ref{analogy}).

In Sect.~\ref{Curvature}, we explore Jordan curvature tensor of \textit{formally real Jordan algebras} (also called \textit{Euclidean Jordan algebras} in the literature), which are regarded
as the counterparts of compact real forms in complex semisimple Lie algebras (see Appendix~\ref{frja}). We first formulate an inequality on formally real Jordan algebras (Lemma~\ref{inequality}), which is of independent
interest. Then we show that every formally real Jordan algebra admits a metric of non-positive
Jordan curvature (Thm.~\ref{nonpositive}).

In Sect.~\ref{Ricci}, we investigate Jordan-Einstein metrics on  Jordan algebras. We show that  simple formally
real Jordan algebras of  dimension at least  two admit a Jordan-Einstein metric  of negative Jordan scalar curvature (Thm.~\ref{JEM}  and Cor.~\ref{Corofss}). For  nontrivial nilpotent Jordan algebras, we prove that they  admit no Jordan-Einstein metrics (Thm.~\ref{nilpotent}). Besides,  we construct a non-nilpotent Jordan algebra of the
semidirect form which admits a flat Jordan-Einstein metric (Ex.~\ref{semidirect}). Moreover, we show that there exists a two-dimensional complex semisimple Jordan algebra, which has   two real forms;  one admits a flat Jordan-Einstein metric, and  the other  admits a  Jordan-Einstein metric of positive Jordan scalar curvature (Ex.~\ref{Ric0+}).

In Sect.~\ref{Summary}, we collect some natural questions  for further study.
Moreover, for the use of this paper we also summarize some basic results of  left-invariant Riemnnian metrics on Lie groups in  Appendix~\ref{Appendix-Lie} and    formally real Jordan algebras in Appendix~\ref{frja}, respectively.

\section{Preliminaries}\label{Basic}
In this section, we recall some basic results of
the moment map  and the Jordan algebras, respectively. The ambient field is always assumed
to be the real number field $\mathbb{R}$ unless otherwise stated.

\subsection{The moment map in algebras}

Let $\mathbb{R}^n$ be the usual $n$-dimensional real vector space. Denote by  $V_n=\otimes^2(\mathbb{R}^n)^* \otimes \mathbb{R}^n$  the space of all bilinear maps (all binary algebras), and
$$
\mathfrak{M}_n=\{\langle\cdot, \cdot\rangle:\langle\cdot, \cdot\rangle \text { is an inner product on } \mathbb{R}^n\}
$$
 the moduli space of all inner products on $\mathbb{R}^n$, respectively. Consider the natural action of $\textnormal{GL}(n)=\textnormal{GL}(\mathbb{R}^n)$ on $V_n$ as follows
\begin{align}\label{G-action}
g.\mu(X, Y)=g \mu(g^{-1} X, g^{-1} Y), ~~\forall g \in \textnormal{GL}(n), \mu \in V_n, X, Y \in \mathbb{R}^n.
\end{align}
Then one immediately sees that  the orbit $\textnormal{GL}(n).\mu$ is precisely the isomorphism class of $\mu$. Differentiating (\ref{G-action}), we obtain the natural representation  $\pi$ of $\mathfrak{gl}(n)$ on $V_n$, i.e.,
\begin{align}\label{g-action}
(\pi(A) \mu)(X, Y)=A \mu(X, Y)-\mu(\pi(A) X, Y)-\mu(X, \pi(A) Y), ~~\forall A \in \mathfrak{gl}(n) .
\end{align}
It follows that $\pi(A) \mu=0$ if and only if $A \in \operatorname{Der}(\mu)$, that is, the derivation algebra of $\mu$. On the other hand, one knows that the linear group $\textnormal{GL}(n)$ also naturally acts on $\mathfrak{M}_n$, i.e.,
\begin{align*}
g.\langle\cdot, \cdot\rangle=\langle g^{-1}(\cdot), g^{-1}(\cdot)\rangle,~~ g \in \textnormal{GL}(n),
\end{align*}
and this action is obviously transitive.

Using the notations above,  it is not hard to verify that the  map
\begin{align*}
g^{-1}: (g.\mu, \langle\cdot, \cdot\rangle)\rightarrow (\mu, g^{-1}.\langle\cdot, \cdot\rangle), ~~\forall g \in \textnormal{GL}(n), \mu \in V_n,
\end{align*}
is an \textit{isometry} between the two metric algebras, that is, preserving the brackets and the inner products simultaneously.
This  implies the  subtle idea: \textit{varying brackets instead of metrics for the study of metric
algebras}, which was  introduced by Lauret in \cite{lauret2001}, and has profoundly influenced contemporary research in homogeneous Riemannian geometry (see \cite{BL2023,Lauret2010,Lauret2011}).

In the sequel, we fix an inner product $\langle\cdot, \cdot\rangle$ on $\mathbb{R}^n$ as a background metric. This  makes each $\mu \in V_n$
become a metric algebra. The fixed inner product  $\langle\cdot, \cdot\rangle$ on $\mathbb{R}^n$   also induces an $\textnormal{O}(n)$-invariant inner product on
$V_n$ as follows
\begin{align}\label{metric}
\langle\mu,\lambda\rangle=\sum_{i,j,k}\langle\mu(E_i,E_j),X_{k}\rangle\langle\lambda(E_i,E_j),E_{k}\rangle,\quad~\mu,\lambda\in  V_n,
\end{align}
where $\left\{E_1, E_2, \cdots, E_n\right\}$ is an arbitrary orthonormal basis of $(\mathbb{R}^n,\langle\cdot, \cdot\rangle)$. Moreover,  there is a natural  inner product on $\mathfrak{gl}(n)$, i.e.,
\begin{align}\label{gl-metric}
(A, B)=\operatorname{Tr} A B^t, ~~A, B \in \mathfrak{gl}(n)
\end{align}
which is $\textnormal{Ad}(\textnormal{O}(n))$-invariant, and the symbol $\cdot^t$ denotes the transpose  relative to $(\mathbb{R}^n,\langle\cdot, \cdot\rangle)$. Decompose $\mathfrak{gl}(n)=\mathfrak{so}(n)+\operatorname{sym}(n)$  into the direct sum of skew matrices and symmetric matrices, respectively, where
\begin{align*}
\mathfrak{so}(n)=\{A \in \mathfrak{g l}(n): A^t=-A\}, \quad \operatorname{sym}(n)=\{A \in \mathfrak{gl}(n): A^t=A\}.
\end{align*}
Then the function $m: V_n \backslash\{0\} \rightarrow \operatorname{sym}(n)$ defined by
\begin{align}\label{MomentDef}
(m(\mu), A)=\frac{\langle\pi(A) \mu, \mu\rangle}{\|\mu\|^2}, ~~~~ 0 \neq \mu \in V_n, A \in \operatorname{sym}(n),
\end{align}
is called the \textit{moment map} for the representation $V_n$ of $\mathfrak{gl}(n)$. Notice that in the complex case, the function $m$ is precisely the moment map from symplectic geometry, corresponding to the Hamiltonian action of $\textnormal{U}(n)$ on the symplectic manifold $\mathbb{P} V_n$ (see \cite{MF94}).

Now, for each $\mu \in V_n$, we associate it  a map  $\textnormal{M}_\mu \in \operatorname{sym}(n)$ as follows
\begin{align}\label{M}
\textnormal{M}_\mu=\sum_{i=1}^n L_{E_i}^\mu(L_{E_i}^\mu)^t-\sum_{i=1}^n(L_{E_i}^\mu)^t L_{E_i}^\mu-\sum_{i=1}^n(R_{E_i}^\mu)^t R_{X_i}^\mu,
\end{align}
where $\left\{E_1, E_2, \cdots, E_n\right\}$ is an arbitrary orthonormal basis of $(\mathbb{R}^n,\langle\cdot, \cdot\rangle)$,  and for any $X\in \mathbb{R}^n$  the  operators $L_X^\mu$, $R_X^\mu: \mathbb{R}^n \rightarrow \mathbb{R}^n$ are given  by $L_X^\mu(Y)=\mu(X, Y)$ and $R_X^\mu(Y)=\mu(Y, X), \forall Y \in \mathbb{R}^n$, respectively.
\begin{Lemma}[\cite{Lauret2011,ZY}]\label{M-formu}
 Let the notations be as above. Then $m(\mu)=\frac{\textnormal{M}_\mu}{\|\mu\|^2}$ for any $0 \neq \mu \in V_n$, and moreover,
\begin{align}
\langle\textnormal{M}_\mu X, Y\rangle= & \sum_{i, j=1}^n\langle\mu(E_i, E_j), X\rangle\langle\mu(E_i, E_j), Y\rangle-\sum_{i, j=1}^n\langle\mu(E_i, X), E_j\rangle\langle\mu(E_i, Y), E_j\rangle \notag\\
& -\sum_{i, j=1}^n\langle\mu(X, E_i), E_j\rangle\langle\mu(Y, E_i), E_j\rangle,~~ \forall X, Y \in \mathbb{R}^n, \label{Mf}
\end{align}
where $\left\{E_1, E_2, \cdots, E_n\right\}$ is an arbitrary orthonormal basis of $(\mathbb{R}^n,\langle\cdot, \cdot\rangle)$.
\end{Lemma}

\begin{remark}
By Lemma~\ref{M-formu}, we know that if $(\mu,\langle\cdot,\cdot\rangle)$ is a metric Lie algebra, then the map $\textnormal{M}_\mu$ in (\ref{Mf}) coincides up to a scalar with the map $\textnormal{M}$ in (\ref{Riccif}) for $(\mu,\langle\cdot, \cdot\rangle)$. As we shall see later, this also holds for metric Jordan algebras (see Remark~\ref{operator}).
\end{remark}

\subsection{Jordan algebras}  An algebra $\mathscr{A}$ is said to be a Jordan algebra if, for all $X, Y$ in $\mathscr{A}$ :
\begin{align}
XY&=YX, \label{commu}\\
X(X^2 Y)&=X^2(XY). \label{weakass}
\end{align}
Using the notation $L_X(Y)=XY$ for all $X, Y \in \mathscr{A}$, and $[S, T]=S T-T S$ for any two endomorphisms of $\mathscr{A}$, the property (\ref{weakass}) can be written $[L_X, L_{X^2}]=0$ for all $X \in \mathscr{A}$.

\begin{remark}
If $(V, \diamond)$ is an associative algebra, then one can naturally define on $V$ a Lie algebra structure and a Jordan algebra structure, respectively,  as follows
\begin{align*}
[X,Y]:=X\diamond Y-Y\diamond X, ~~\forall X, Y \in V,\quad
XY:=\frac{1}{2}(X\diamond Y+Y\diamond X), ~~\forall X, Y \in V.
\end{align*}
In general a Jordan algebra is not associative.
\end{remark}
A Jordan algebra $\mathscr{A}$ is called \textit{simple} if $\mathscr{A}^2 \neq 0$ and $\mathscr{A}$ has no nontrivial ideals, and it is called \textit{semisimple} if it is a direct product of simple Jordan algebras.

\begin{Definition}
 Let $\mathscr{A}$ be a Jordan algebra. The Jordan algebra $\mathscr{A}$ is called \textit{solvable} if $\mathscr{A}^{(m)}=0$ for some $m \in \mathbb{N}$, where $\mathscr{A}^{(0)}=\mathscr{A}, \mathscr{A}^{(k+1)}=\mathscr{A}^{(k)} \mathscr{A}^{(k)}, k \geq 0$. The Jordan algebra $\mathscr{A}$ is called \textit{nilpotent} if $\mathscr{A}^m=0$ for some $m \in \mathbb{N}$, where $\mathscr{A}^0=\mathscr{A}, \mathscr{A}^{k+1}=\mathscr{A} \mathscr{A}^k, k \geq 0$.
\end{Definition}
Unlike in the context  of Lie algebras, the concepts of solvability and nilpotency in Jordan algebras turn out to be equivalent (see \cite{Albert1946}).

\begin{Lemma}[\cite{FK1994}]\label{tau}
 Let $\mathscr{A}$ be a Jordan algebra, then the following symmetric bilinear form
\begin{align}\label{tauf}
\tau(X, Y):=\operatorname{Tr} L_{XY}, \forall X, Y \in \mathscr{A},
\end{align}
is associative, that is, $\tau(X Y, Z)=\tau(X, Y Z)$ for all $X, Y, Z \in \mathscr{A}$.
\end{Lemma}

 For a Jordan algebra $\mathscr{A}$,  we denote by $\mathscr{N}$  the \textit{radical}, i.e., the maximal nilpotent ideal of $\mathscr{A}$.  It is proved by A. Albert that the radical $\mathscr{N}$ of  $\mathscr{A}$   coincides with the kernel of the symmetric bilinear form $\tau$ in (\ref{tauf}).

 Moreover, we have the following theorem

\begin{theorem}[Wedderburn Principal Theorem] \label{WPT}
Any real Jordan algebra $\mathscr{A}$ can be written as a vector space direct sum $\mathscr{A}=\mathscr{S}+\mathscr{N}$, where $\mathscr{N}$ is the radical of $\mathscr{A}$, and $\mathscr{S}$ is a maximal semisimple subalgebra of $\mathscr{A}$ isomorphic to $\mathscr{A}/\mathscr{N}$.
\end{theorem}
It follows  Theorem~\ref{WPT} that $\mathscr{A}$ is semisimple if and only if the   bilinear form $\tau$ is non-degenerate. By this result  and Lemma~\ref{tau}, one can easily prove that every semisimple Jordan algebra necessarily  carries  an identity element.

Now, let $\mathscr{A}$ be a Jordan algebra with an identity element $E$, and $\mathbb{R}[Y]$ denote the polynomials in the variable $Y \in \mathscr{A}$ with coefficients in $\mathbb{R}$. Notice that  $\mathbb{R}[Y]$ is the subalgebra of $\mathscr{A}$ generated by $Y$ and the identity element,   the number $\operatorname{dim} \mathbb{R}[Y]$ is necessarily   bounded. We define the \textit{rank} of $\mathscr{A}$ as follows
\begin{align}\label{rank}
r=\max\{\operatorname{dim} \mathbb{R}[Y]: Y \in \mathscr{A}\}.
\end{align}
An element of $X \in \mathscr{A}$ is called \textit{regular}, if $\operatorname{dim} \mathbb{R}[X]$ is maximal in $\{\operatorname{dim} \mathbb{R}[Y]: Y \in \mathscr{A}\} $. The \textit{reduced trace} of a regular element $X \in \mathscr{A}$ is defined by
\begin{align}\label{tr}
\operatorname{tr}(X):=\operatorname{Tr} {L_X|_{\mathbb{R}[X]}}.
\end{align}
Since the set of regular elements is dense in $\mathscr{A},$ the function (\ref{tr}) can be uniquely extended to a linear map  $\operatorname{tr}:\mathscr{A} \rightarrow \mathbb{R},$ and in particular,  we have  $\operatorname{tr}(E)=r$ (see \cite{FK1994}).

\begin{Lemma}[\cite{FK1994}]\label{trass}
Let $\mathscr{A}$ be a Jordan algebra with an identity element. Then the symmetric bilinear form $\operatorname{tr}(X Y)$ is associative, that is,
 $\operatorname{tr}((XY)Z)=\operatorname{tr}(X(Y Z))$ for all $X, Y, Z \in \mathscr{A}$.
\end{Lemma}

For a Jordan algebra $\mathscr{A}$, the Killing form on $\mathscr{A}$ is defined as follows
\begin{align}\label{Killing}
B(X, Y):=\operatorname{Tr} L_X L_Y,
\end{align}
for all $X, Y \in \mathscr{A}$. We note that the Killing form of a Jordan algebra is symmetric, but in general not associative.

\section{Connections and curvatures on  Jordan algebras}\label{Connection}
In this section, we  first  introduce the  concepts of {connection} and curvature on Jordan algebras, then  obtain  the corresponding   formulas. Moreover, we  get a  Jordan version of Theorem~\ref{CL}.

\subsection{Jordan-Levi-Civita Connection}
\begin{Definition}\label{JLCC}
Let $\mathscr{A}$ be a  Jordan algebra and $\langle\cdot, \cdot\rangle$ be an inner product. For any bilinear map
\begin{align*}
\nabla: \mathscr{A} \times \mathscr{A} \rightarrow \mathscr{A},
\end{align*}
write $\nabla_X Y:=\nabla(X, Y)$ for any $X, Y \in \mathscr{A}$.  We call such $\nabla$ a \textit{connection} on $\mathscr{A}$.  If the connection $\nabla$ satisfies
\begin{align}\label{c}
\nabla_XY&+\nabla_YX=XY, ~~\forall X, Y\in \mathscr{A},
\end{align}
then we call $\nabla$ a \textit{torsion-free connection} on $\mathscr{A}$. If $\nabla$ is   {torsion-free} and   satisfies
\begin{align}\label{metric-sym}
\langle\nabla_X Y, Z&\rangle-\langle Y, \nabla_X Z\rangle=0, ~~\forall X, Y\in \mathscr{A},
\end{align}
then we call $\nabla$ a \textit{Jordan-Levi-Civita connection} on $(\mathscr{A},\langle\cdot, \cdot\rangle)$.
\end{Definition}

\begin{remark}
Invariant super-connections on Lie super-groups have two components: the even one and the odd one (\cite{Go}). The  common  connection on metric Lie algebras  corresponds to  the even one (see Appendix~\ref{Appendix-Lie}), and  our  Definition~\ref{JLCC} is  formally   adapted    to the  odd one  of an invariant super-connection. Indeed, suppose that $G$ be a  Lie super-group with the Lie super-algebra $\g=\g_0+\g_1$ consisting
of left-invariant vector fields, and $\langle\cdot,\cdot\rangle$ be a left-invariant super Riemannian metric on $G$. Then we may identify $(G,\langle\cdot,\cdot\rangle)$ with $(\g,\langle\cdot,\cdot\rangle)$.
Let  $\nabla$ be the Levi-Civita connection of $(\g,\langle\cdot,\cdot\rangle)$,  and $X, Y, Z \in \g$  be homogeneous elements with  $|X|\in \{0,1\}$ denoting the parity of $X.$ Then the torsion-free property and the
metric-preserving property of $\nabla$ are respectively given by
\begin{align}
\nabla_XY-(-1)^{|X||Y|}\nabla_YX-[X,Y]=0, \label{stor-f}\\
\langle \nabla_XY, Z\rangle+(-1)^{|X||Y|}\langle Y,\nabla_XZ\rangle=0.\label{smetric-skew}
\end{align}
Similar to the standard theory in Riemannian geometry, the Levi-Civita connection $\nabla$  is uniquely determined by
 (\ref{stor-f}) and (\ref{smetric-skew}), i.e.,
\begin{align}\label{skoz}
2\langle \nabla_XY,&Z\rangle=\langle [X,Y],Z\rangle)-(-1)^{|X|(|Y|+|Z|)}\langle [Y,Z],X\rangle)+(-1)^{|Z|(|X|+|Y|)}\langle [Z,X],Y\rangle.
\end{align}
Moreover,  associated with the Levi-Civita connection , the curvature  $(\g,\langle\cdot,\cdot\rangle)$ is given by
\begin{align}\label{srrc}
R(X,Y)Z=-\nabla_X\nabla_YZ+(-1)^{|X||Y|}\nabla_Y\nabla_XZ+\nabla_{[X,Y]}Z.
\end{align}
We refer to \cite{Go}  for  more details about invariant structures on Lie super-groups,.
\end{remark}

\begin{remark}
Clearly, for a metric Lie algebra $(\mathfrak{g},\langle\cdot, \cdot\rangle)$, the property (\ref{metric-skew}) is equivalent to $\nabla_X$ being skew-symmetric for all $X \in \mathfrak{g}$, while for a metric Jordan algebra $(\mathscr{A},\langle\cdot, \cdot\rangle)$, the equation (\ref{metric-sym}) is equivalent to $\nabla_X$ being symmetric for all $X \in \mathscr{A}$.
\end{remark}

The following theorem might be regarded as a special case of (\ref{skoz}) within the context of invariant metrics on Lie super-groups.

\begin{theorem}\label{unique}
Every metric Jordan algebra $(\mathscr{A},\langle\cdot, \cdot\rangle)$ admits a unique Jordan-Levi-Civita connection $\nabla$, which is given by
\begin{align*}
\langle\nabla_X Y, Z\rangle=\frac{1}{2}(\langle X Y, Z\rangle-\langle Y Z, X\rangle+\langle Z X, Y\rangle),
\end{align*}
for all $X, Y, Z \in \mathscr{A}$.
\end{theorem}
\begin{proof}
 Assume that $\nabla$ a Jordan-Levi-Civita connection on $(\mathscr{A},\langle\cdot, \cdot\rangle)$, then
 \begin{align}
\langle\nabla_X Y, Z\rangle-\langle Y, \nabla_X Z\rangle=0, \label{1} \\
\langle\nabla_Y Z, X\rangle-\langle Z, \nabla_Y X\rangle=0, \label{2} \\
\langle\nabla_Z X, Y\rangle-\langle X, \nabla_Z Y\rangle=0, \label{3}
 \end{align}
for any $X, Y, Z \in \mathscr{A}$. Adding (\ref{1}) and (\ref{2}) and subtracting (\ref{3}), we have
\begin{align*}
\langle\nabla_X Y-\nabla_Y X, Z\rangle-\langle\nabla_Z X+\nabla_X Z, Y\rangle+\langle\nabla_Y Z+\nabla_Z Y, X\rangle=0.
\end{align*}
Using the property (\ref{c}), we have
\begin{align}\label{Koz}
\langle\nabla_X Y, Z\rangle=\frac{1}{2}(\langle X Y, Z\rangle-\langle Y Z, X\rangle+\langle Z X, Y\rangle).
\end{align}
So the Jordan-Levi-Civita connection $\nabla$ is uniquely determined by (\ref{c}) and (\ref{metric-sym}).

To prove the existence, we define $\nabla$ by (\ref{Koz}). It is easy to verify that $\nabla$ is well-defined and that it satisfies the desired conditions (\ref{c}) and (\ref{metric-sym}).
\end{proof}

\subsection{The Jordan Curvature Tensor}
\begin{Definition}\label{JCT}
Let $(\mathscr{A},\nabla)$ be a  Jordan algebra with  a  connection. The  \textit{curvature tensor}  of $(\mathscr{A},\nabla)$ is defined by
\begin{align*}
R(X, Y) Z=-\nabla_X \nabla_Y Z-\nabla_Y \nabla_X Z+\nabla_{X Y} Z,~~X, Y, Z \in \mathscr{A},
\end{align*}
or equivalently, $R(X, Y)=\nabla_{X Y}-(\nabla_X \nabla_Y+\nabla_Y \nabla_X)$.
If moreover, $\nabla$ comes from   the Jordan-Levi-Civita connection of a  metric Jordan algebra $(\mathscr{A},\langle\cdot, \cdot\rangle)$, then we call $R$ the \textit{Jordan curvature tensor}.
\end{Definition}

The  Definition~\ref{JCT} is  formally   adapted    to  the curvature restricted to the  odd component  of  super-connections (\ref{srrc}).

\begin{remark}
 It is easily seen that $R$ is symmetric in the first two positions. Note that if  $R=0$, also referred to as \textit{flat},  then $\nabla$ naturally induces a Jordan algebra representation of $\mathscr{A}$.
\end{remark}

\begin{remark}\label{2dim}
Consider the two-dimensional Jordan algebra $$\mathscr{A}: e_1 e_1=e_1, e_1 e_2=e_2.$$ Endow $\mathscr{A}$ with the metric $\langle\cdot,\cdot\rangle$ so that $\{e_1, e_2\}$ is an orthonormal basis. Then by a straightforward calculation, we know that the Jordan-Levi-Civita connection satisfies $\nabla_{e_1} e_1=e_1, \nabla_{e_1} e_2=e_2$ and $\nabla_{e_2} e_1=0=\nabla_{e_2} e_2$. It follows that the only non-trivial term $R(e_i, e_j)e_k$ is $R(e_1,e_1)e_2=-e_2$. So $R(e_1, e_1)e_2+R(e_1, e_2)e_1
+R(e_2, e_1)e_1=-e_2$. This shows that the sum $R(X, Y)Z+$ $R(Y, Z)X+R(Z, X)Y$ in general does not vanish.
\end{remark}

Now, we explore Theorem~\ref{CL} in the context of  Jordan algebra. Suppose that $(\mathscr{A},\nabla)$ is a Jordan algebra with a flat, torsion-free connection, that is,
\begin{align*}
\nabla_X \nabla_Y Z+\nabla_Y \nabla_X Z-\nabla_{X Y} Z=0,\quad \nabla_XY+\nabla_YX=XY
\end{align*}
for all $X, Y, Z \in \mathscr{A}$. Define on $\mathscr{A}$ a new binary operation
\begin{align*}
X\diamond Y:=\nabla_XY,~~\forall X, Y\in \mathscr{A},
\end{align*}
then it is easy to verify that $\diamond$ satisfies
\begin{align*}
(X\diamond Y)\diamond Z-X\diamond(Y\diamond Z)=-(Y\diamond X)\diamond Z+Y\diamond(X\diamond Z),~~\forall X,Y,Z\in \mathscr{A}.
\end{align*}
Note that this differs from the equation (\ref{+}) by a sign. In an analogy of left symmetric algebra (\ref{+}), we introduce the following definition
\begin{Definition}\label{lssa}
Let $(V,\diamond)$ is an \textit{arbitrary} binary algebra. If $(V,\diamond)$  satisfies
\begin{align}\label{ZFD}
(X\diamond Y)\diamond Z-X\diamond(Y\diamond Z)=-(Y\diamond X)\diamond Z+Y\diamond(X\diamond Z),~~\forall X,Y,Z\in V,
\end{align}
then we call $(V,\diamond)$ is a \textit{left skew-symmetric algebra} (write LSSA for short).
\end{Definition}

Clearly, an algebra is   both  a LSA (i.e., satisfying the equation  (\ref{+})) and  a LSSA if and only if it is an associative algebra.

\begin{theorem}\label{ZFDCDS}
Let $(V,\diamond)$ is a left skew-symmetric algebra.  Then the new binary operation
\begin{align}\label{sub}
X\bullet Y:=X\diamond Y+Y\diamond X,~~\forall X, Y\in V
\end{align}
defines on $V$ a Jordan structure.
\end{theorem}
\begin{proof}
Note that for any $X, Y\in V$
\begin{align*}
\textbf{J}&=\frac{1}{2}X\bullet ((X\bullet X)\bullet Y)-\frac{1}{2}(X\bullet X)\bullet (X\bullet Y)\\
&=X\bullet ((X\diamond X)\bullet Y)-(X\diamond X)\bullet (X\bullet Y)\\
&=X\bullet ((X\diamond X)\diamond Y+Y\diamond (X\diamond X))-(X\diamond X)\bullet (X\diamond Y+Y\diamond X)\\
&=X\diamond ((X\diamond X)\diamond Y+Y\diamond (X\diamond X))+((X\diamond X)\diamond Y+Y\diamond (X\diamond X))\diamond X\\
&\quad-(X\diamond X)\diamond (X\diamond Y+Y\diamond X)-(X\diamond Y+Y\diamond X)\diamond(X\diamond X)
\end{align*}
Put
\begin{align*}
a&=X\diamond ((X\diamond X)\diamond Y),\quad b=X\diamond(Y\diamond (X\diamond X)),\quad c= ((X\diamond X)\diamond Y)\diamond X,\quad d=(Y\diamond (X\diamond X))\diamond X,\\
e&=(X\diamond X)\diamond(X\diamond Y),\quad f=(X\diamond X)\diamond(Y\diamond X),\quad g=(X\diamond Y)\diamond(X\diamond X),\quad h=(Y\diamond X)\diamond(X\diamond X),
\end{align*}
then $\textbf{J}=a+b+c+d-e-f-g-h.$ Using (\ref{ZFD}), we have
\begin{align*}
b-g&=-Y\diamond(X\diamond (X\diamond X))+(Y\diamond X)\diamond(X\diamond X)=-Y\diamond(X\diamond (X\diamond X))+h,\\
c-f&=-(Y\diamond (X\diamond X))\diamond X+Y\diamond((X\diamond X)\diamond X)=Y\diamond((X\diamond X)\diamond X)-d.
\end{align*}
Since $((X\diamond X)\diamond Z)-X\diamond (X\diamond Z)=-((X\diamond X)\diamond Z)+X\diamond (X\diamond Z)=0$, then $(X\diamond X)\diamond Z=X\diamond (X\diamond Z)$  for any $X,Z\in V$. In particular, $X\diamond (X\diamond X)=(X\diamond X)\diamond X$, and
\begin{align*}
a-e&=X\diamond ((X\diamond X)\diamond Y)-(X\diamond X)\diamond(X\diamond Y)=X\diamond (X\diamond (X\diamond Y))-(X\diamond X)\diamond(X\diamond Y)\\
&=(X\diamond X)\diamond(X\diamond Y)-(X\diamond X)\diamond(X\diamond Y)=0.
\end{align*}
Consequently, $\textbf{J}=a+b+c+d-e-f-g-h=0,$ that is, (\ref{sub}) defines a Jordan structure.
\end{proof}

\begin{Corollary}\label{analogy}
For a Jordan algebra $\mathscr{A},$  there is a one-to-one correspondence between    flat torsion-free  connections  on $\mathscr{A}$ and  LSSA-structures on $\mathscr{A}$.
\end{Corollary}

In the sequel, we always assume that $\nabla$, $R$  are  the Jordan-Levi-Civita connection and
the  Jordan curvature tensor of  the metric Jordan algerba $(\mathscr{A},\langle\cdot, \cdot\rangle)$, respectively.  We  also write $R(X,Y,Z,W)=\langle R(X,Y)Z,W\rangle$ for any  $X, Y, Z, W\in \mathscr{A}$

\begin{Lemma}
Let $(\mathscr{A},\langle\cdot, \cdot\rangle)$ be a metric Jordan algebra. Then
\begin{enumerate}
\item [(a)] $R(X,Y,Z,W)=R(Y,X,Z,W);$
\item [(b)] $R(X,Y,Z,W)=R(X,Y,W,Z).$
\end{enumerate}
\end{Lemma}
\begin{proof}
Obviously, (a) holds. For (b), it follows from that $\nabla_{X}$ is symmetric for all $X\in\mathscr{A}.$
\end{proof}

We note that in general $R(X, Y, Z, W) \neq R(Z, W, X, Y)$, see Remark~\ref{2dim}.
\begin{Definition}\label{JC}
The Jordan curvature of $(\mathscr{A},\langle\cdot,\cdot\rangle)$ is defined as follows
\begin{align*}
\mathcal{J}(X,Y)=\frac{\langle R(X, Y) X, Y\rangle}{\langle X, X\rangle\langle Y, Y\rangle-\langle X, Y\rangle^2}
\end{align*}
for any linearly independent $X, Y \in \mathscr{A}$. If moreover, $\mathcal{J}$ is a constant, $(\mathscr{A},\langle\cdot, \cdot\rangle)$ is called of constant Jordan curvature.
\end{Definition}
\begin{exam}
Consider the Jordan algebras
\begin{align*}
\mathscr{A}:~&e_1e_1=e_1, e_2e_2=e_2, \cdots, e_ne_n=e_n; \\
\mathscr{B}:~&e_1e_1=e_1, e_1e_2=\frac{1}{2}e_2, \cdots, e_1e_n=\frac{1}{2}e_n.
\end{align*}
It is easy to verify that $\mathscr{A}$ and  $\mathscr{B}$  both admit a metric of constant zero Jordan curvature.
\end{exam}

\subsection{The Jordan Ricci Curvature tensor}
For a metric Jordan algebra $(\mathscr{A},\langle\cdot, \cdot\rangle)$, one may define the Jordan Ricci curvature tensor as follows
\begin{align}\label{ric}
\operatorname{ric}(U, V)=\operatorname{Tr}{(X \mapsto R(U, X) V)},~~\forall U, V \in \mathscr{A}.
\end{align}
We point out that the tensor $\operatorname{ric}$ is in general not symmetric. It is thus natural to introduce the following (symmetric) Jordan Ricci curvature tensor

\begin{Definition}\label{JRT}
The Jordan Ricci curvature tensor of $(\mathscr{A},\langle\cdot, \cdot\rangle)$ is defined as follows
\begin{align}\label{SJTR}
\operatorname{Ric}(U, V)=\frac{1}{2}(\operatorname{ric}(U, V)+\operatorname{ric}(V, U)), \forall U, V \in \mathscr{A}.
\end{align}
If moreover, $\operatorname{Ric}$ is a constant multiple of $\langle \cdot,\cdot\rangle$, then we call $\langle\cdot,\cdot\rangle$ a Jordan-Einstein metric on $(\mathscr{A},\langle\cdot, \cdot\rangle)$.
\end{Definition}
As we shall see later, \textnormal{ric} and \textnormal{Ric} are closely related to each other by the following so called \textit{mean curvature vector}.

\begin{Definition}\label{MCV}
For a metric Jordan algebra $(\mathscr{A},\langle\cdot, \cdot\rangle)$, we define $H \in \mathscr{A}$ by
\begin{align*}
H=\sum_{i=1}^{\dim \mathscr{A}} E_i^2,
\end{align*}
where $\{E_i\}$ is an orthonormal basis of $(\mathscr{A},\langle\cdot, \cdot\rangle)$,  and we  call $H$   the \textit{mean curvature vector} of $(\mathscr{A},\langle\cdot, \cdot\rangle)$.
\end{Definition}
It is easily seen that the mean curvature vector is  independent of the choice of  orthonormal bases.

Following \cite{Dotti}, for any $U \in \mathscr{A}$, we define two linear maps $\nabla_U$ and $\nabla U$ on $(\mathscr{A},\langle\cdot, \cdot\rangle)$, respectively,  as follows
 \begin{align*}
\nabla_U(X):=\nabla_UX,\quad \nabla U(X):=\nabla_XU,
 \end{align*}
for any $X\in\mathscr{A}$.

\begin{Lemma}\label{JRTrF}
The Jordan Ricci tensor of $(\mathscr{A},\langle\cdot, \cdot\rangle)$ satisfies
\begin{align*}
\textnormal{Ric}(U,U)=\operatorname{Tr}{(\nabla U)^2}-\frac{1}{2}\operatorname{Tr}{\nabla U^2},
\end{align*}
for any $U\in\mathscr{A}$.
\end{Lemma}

\begin{proof}
Assume that $\{E_i\}$ is an orthonormal basis of  $(\mathscr{A},\langle\cdot, \cdot\rangle)$, then we have
\begin{align*}
\textnormal{Ric}(U,U)=\operatorname{Tr}{(X \mapsto R(U, X)U)}=\sum_i \langle R(U, E_i)U, E_i\rangle=\sum_i \langle R(E_i,U)U, E_i\rangle.
\end{align*}
It follows that
\begin{align*}
\operatorname{Ric}(U, U) & =\sum_i\langle-\nabla_{E_i} \nabla_UU-\nabla_U \nabla_{E_i} U+\nabla_{{E_i}U}U, E_i\rangle \\
& =\sum_i\langle-\nabla_U \nabla_{E_i} U+\nabla_{{E_i}U}U, E_i\rangle-\sum_i\langle\nabla_{E_i} \nabla_UU, E_i\rangle \\
& =: \textnormal{I}-\textnormal{II}.
\end{align*}
Note that $E_i U=\nabla_{E_i} U+\nabla_U E_i$, then
\begin{align*}
\mathrm{I} & =\sum_i\langle-\nabla_U \nabla_{E_i} U, E_i\rangle+\sum_i\langle\nabla_{\nabla_U E_i} U, E_i\rangle+\sum_i\langle\nabla_{\nabla_{E_i} U} U, E_i\rangle \\
& =-\operatorname{Tr} \nabla_U \circ \nabla U+\operatorname{Tr} \nabla U \circ \nabla_U+\operatorname{Tr} \nabla U \circ \nabla U \\
& =\operatorname{Tr}{(\nabla U)^2}.
\end{align*}
Since $\nabla_U U+\nabla_U U=U^2$, then $$\textnormal{II}=\sum_i(\nabla_{E_i} \nabla_U U, E_i\rangle=\frac{1}{2} \sum_i(\nabla_{E_i} \nabla U^2, E_i\rangle=\frac{1}{2} \operatorname{Tr} \nabla U^2.$$
This completes the proof.
\end{proof}

\begin{remark}\label{ricTrf}
 It follows from a similar calculation that
\begin{align*}
\operatorname{ric}(U, V)=\operatorname{Tr} \nabla U\circ\nabla V-\operatorname{Tr} \nabla \nabla_UV,
\end{align*}
for any $U, V \in \mathscr{A}$.
\end{remark}

\begin{theorem}\label{RF}
Let $\left\{E_i\right\}$ be an orthonormal basis of $(\mathscr{A},\langle\cdot,\cdot\rangle).$ Then
\begin{align*}
\operatorname{Ric}(X, Y)=-&\frac{1}{2} \sum_{i, j}\langle XE_i, E_j\rangle\langle YE_i, E_j\rangle+\frac{1}{4}\sum_{i, j}\langle E_iE_j, X\rangle\langle E_iE_j, Y\rangle \\
& +\frac{1}{2}B(X, Y)-\frac{1}{4}\langle H, XY\rangle,~~\forall X, Y \in \mathscr{A},
\end{align*}
where $B$ is the Killing form of $\mathscr{A}, H$ is the mean curvature vector of $(\mathscr{A},\langle\cdot,\cdot\rangle)$.
\end{theorem}

\begin{proof}
Let $\{E_i\}$ be an orthonormal basis of $(\mathscr{A},\langle\cdot, \cdot\rangle)$. Since by Theorem~\ref{unique}
\begin{align}\label{k}
\langle\nabla_{E_i} U, E_j\rangle=\frac{1}{2}(\langle UE_i, E_j\rangle-\langle UE_j, E_i\rangle+\langle E_iE_j, U\rangle),
\end{align}
for all $i, j$, and $U \in \mathscr{A}$, then
\begin{align*}
\operatorname{Tr}{(\nabla U)^2} & =\sum_j\langle\nabla_{\nabla_{E_j} U}U, E_j\rangle \\
& =\sum_{i, j}\langle\nabla_{E_i}U, E_j\rangle\langle\nabla_{E_j}U, E_i\rangle \\
& =\frac{1}{4} \sum_{i, j}\langle E_iE_j, U\rangle^2-\underbrace{\frac{1}{4} \sum_{i, j}(\langle UE_i, E_j\rangle
-\langle UE_j, E_i\rangle)^2}_{\textnormal{I}}.
\end{align*}
It follows that
\begin{align*}
\textnormal{I} & =\frac{1}{4} \sum_{i, j}\langle UE_i, E_j\rangle^2-\frac{1}{2} \sum_{i, j}\langle UE_i, E_j\rangle\langle UE_j, E_i\rangle+\frac{1}{4} \sum_{i, j}\langle UE_j, E_i\rangle^2 \\
& =\frac{1}{2} \sum_{i, j}\langle UE_i, E_j\rangle^2-\frac{1}{2} \sum_i\langle U(UE_i), E_i\rangle \\
& =\frac{1}{2} \sum_{i, j}\langle UE_i, E_j\rangle^2-\frac{1}{2} \operatorname{Tr} L_UL_U .
\end{align*}
So
\begin{align*}
\operatorname{Tr}{(\nabla U)^2}=\frac{1}{4} \sum_{i, j}\langle E_iE_j, U\rangle^2-\frac{1}{2} \sum_{i, j}\langle UE_i, E_j\rangle^2+\frac{1}{2} B(U, U).
\end{align*}
Note that
\begin{align*}
\operatorname{Tr} {\nabla U^2}=\sum_i\langle\nabla_{E_i} U^2, E_i\rangle=\sum_i\langle U^2, \nabla_{E_i}E_i\rangle=\sum_i\langle U^2, \frac{1}{2} E_i^2\rangle=\frac{1}{2}\langle U^2, H\rangle.
\end{align*}
Then by Lemma~\ref{JRTrF}, we have
\begin{align*}
\operatorname{Ric}(U, U)=\operatorname{Tr}{(\nabla U)^2}-\frac{1}{2} \operatorname{Tr} {\nabla U^2}=-\frac{1}{2} \sum_{i, j}\langle UE_i, E_j\rangle^2+\frac{1}{4}\sum_{i, j}\langle E_iE_j, U\rangle^2+\frac{1}{2}B(U, U)-\frac{1}{4}\langle H, U^2\rangle.
\end{align*}
This completes the proof.
\end{proof}

\begin{remark}\label{operator}
By Theorem~\ref{RF}, the Ricci operator of $(\mathscr{A},\langle\cdot,\cdot\rangle)$ can be written  as follows
\begin{align}\label{Ricciformula}
\textnormal{Ric}=\textnormal{M}+\frac{1}{2}B-\frac{1}{4}S_H,
\end{align}
where \textnormal{M} coincides up to a scalar $\frac{1}{4}$ with the {moment map} for the variety of Jordan algebras (see Lemma~\ref{M-formu}), $B$ is the symmetric operator defined by  the Killing form of $\mathscr{A}$, and $S_H$ is the  symmetric operator defined by $\langle S_H X, Y\rangle=\langle H, XY\rangle,~~\forall X, Y \in\mathscr{A}.$ Note  that if $\mathscr{A}$ is nilpotent, and $H\in (\mathscr{AA})^\perp$ (for example $H=0$), then the Ricci operator coincides with \textnormal{M}.
\end{remark}

\begin{remark}\label{ric=Ric}
Similarly, by Remark~\ref{ricTrf}, one    obtains the following result
\begin{align*}
\operatorname{ric}(X, Y)=- & \frac{1}{2} \sum_{i, j}\langle XE_i, E_j\rangle\langle YE_i, E_j\rangle+\frac{1}{4} \sum_{i, j}\langle E_iE_j, X\rangle\langle E_iE_j, Y\rangle \\
& +\frac{1}{2} B(X, Y)-\frac{1}{2}\langle H, \nabla_X Y\rangle, ~~\forall X, Y \in \mathscr{A}.
\end{align*}
Moreover, $\textnormal{ric}=\textnormal{Ric}$ if and only if the operator $L_H$ is self-adjoint. In particular,  $\textnormal{ric}=\textnormal{Ric}$  if $H=0$ or a constant multiple of the identity element.
\end{remark}

\begin{remark}
 Suppose that $(\mathscr{A},\langle\cdot,\cdot\rangle)$ is of constant Jordan curvature, then $(\mathscr{A},\langle\cdot,\cdot\rangle)$ is necessarily Jordan-Einstein. Indeed, by Definition~\ref{JC} we know that there exists a constant $c$ such that
 \begin{align*}
\langle R(X, Y) X, Y\rangle=c(\langle X, X\rangle\langle Y, Y\rangle-\langle X, Y\rangle^2),
 \end{align*}
for all $X, Y \in \mathscr{A}$. So
\begin{align*}
\operatorname{Ric}(U, U)=\sum_{i=1}^{\dim \mathscr{A}}\langle R(U, E_i)U, E_i\rangle=\sum_{i=1}^{\dim \mathscr{A}} c(\langle U, U\rangle\langle E_i, E_i\rangle-\langle U, E_i\rangle^2)=c(\operatorname{dim}\mathscr{A}-1)\langle U, U\rangle, \forall U \in \mathscr{A}.
\end{align*}
That is, $\langle\cdot, \cdot\rangle$ is a Jordan-Einstein metric. The converse is not true, see Remark~\ref{2dim}.
\end{remark}

\subsection{Jordan scalar curvature}
By Definition~\ref{JRT}, we introduce the Jordan scalar curvature for metric Jordan algebras, which is a trace of $\textnormal{Ric}$, i.e.,

\begin{Definition}\label{JSC}
For a metric Jordan algebra $(\mathscr{A},\langle\cdot,\cdot\rangle)$, the Jordan scalar curvature is $\textnormal{sc}=\textnormal{Tr}_{\langle\cdot,\cdot\rangle}{\textnormal{Ric}}.$
\end{Definition}

By Theorem~\ref{RF}, it is easy to see that the Jordan scalar curvature is given by
\begin{align}\label{scf}
\textnormal{sc}=-\frac{1}{4} \sum_{i, j, k}\langle E_iE_j, E_k\rangle^2+\frac{1}{2}\sum_{i}B(E_i,E_i)-\frac{1}{4}\langle H,H\rangle.
\end{align}
where $\{E_i\}$ is an orthonormal basis of $(\mathscr{A},\langle\cdot,\cdot\rangle)$. It follows that the Jordan scalar curvature \textnormal{sc} of a
nilpotent metric Jordan algebra satisfies $\textnormal{sc}\leq 0$, and the equality holds if and only if $\mathscr{A}$ is trivial, i.e.,
$XY=0$ for any $X, Y \in \mathscr{A}.$

\section{The Jordan curvature of metric Jordan algebras}\label{Curvature}
In this section, we study the Jordan curvature of formally real Jordan algebras, which are counterparts of compact real forms in complex semisimple Lie algebras (see Appendix~\ref{frja} for a comparison). By Proposition~\ref{FRJAC}, a Jordan algebra $\mathscr{A}$ with an identity element is formally real, if and only if $\mathscr{A}$ admits an associative inner product.

\begin{Lemma}\label{1/2}
Let $(\mathscr{A},\langle\cdot, \cdot\rangle)$ be a metric Jordan algebra. Assume that $\langle\cdot, \cdot\rangle$, is associative, then the Jordan-Levi-Civita connection is given by
\begin{align*}
\nabla_XY=\frac{1}{2}XY,~\forall X, Y \in \mathscr{A}.
\end{align*}
Moreover,
\begin{align*}
R(X,Y)Z+R(Y, Z)X+R(Z,X)Y=0,
\end{align*}
for all $X, Y, Z \in \mathscr{A}$.
\end{Lemma}
\begin{proof}
Since $\langle\cdot, \cdot\rangle$ is associative, then
\begin{align*}
\langle X Y, Z\rangle=\langle X, Y Z\rangle, ~~\forall X, Y, Z \in \mathscr{A}.
\end{align*}
By Theorem~\ref{unique}, we have
\begin{align}\label{12xy}
\nabla_X Y=\frac{1}{2} XY, ~~\forall X,Y \in \mathscr{A}.
\end{align}
This proves the first statements. Moreover, it follows from (\ref{12xy}) that
\begin{align*}
R(X,Y)Z&=\nabla_{XY}Z-\nabla_X\nabla_YZ-\nabla_Y\nabla_XZ \\
&=\frac{1}{2}(XY)Z-\frac{1}{2}\nabla_XYZ-\frac{1}{2} \nabla_YXZ \\
&=\frac{1}{2}(XY)Z-\frac{1}{4} X(YZ)-\frac{1}{4}Y(XZ) \\
&=\frac{1}{2}(XY)Z-\frac{1}{4}(YZ)X-\frac{1}{4}(ZX)Y.
\end{align*}
Similarly,
\begin{align*}
R(Y,Z)X&=\frac{1}{2}(YZ)X-\frac{1}{4}(ZX)Y-\frac{1}{4}(XY)Z,\\
R(Z,X)Y&=\frac{1}{2}(ZX)Y-\frac{1}{4}(XY)Z-\frac{1}{4}(YZ)X.
\end{align*}
So
\begin{align*}
R(X,Y)Z+R(Y, Z)X+R(Z,X)Y=0,
\end{align*}
for all $X, Y, Z \in \mathscr{A}$. This completes the proof.
\end{proof}

\begin{Lemma}\label{inequality}
Let $\mathscr{A}$ be a formally real Jordan algebra, and $\langle\cdot, \cdot\rangle$ be an associative inner product on it. Then we have the following inequality
\begin{align}\label{XYXY}
\langle XY, XY\rangle \leq\langle X^2, Y^2\rangle, ~~\forall X, Y \in \mathscr{A}.
\end{align}
\end{Lemma}

\begin{proof}
By Theorem~\ref{FRJSS} and Proposition~\ref{simpleuni}, it suffices to prove the lemma in the case that $\mathscr{A}$ is simple. Now, let $\mathscr{A}$ be a simple formally real Jordan algebra, and $r$ be the rank of $\mathscr{A}$. By the associativity of the inner product and the commutativity of $\mathscr{A}$
\begin{align*}
\langle XY, XY\rangle=\langle X(XY), Y\rangle=\langle L_X^2(Y), Y\rangle,
\end{align*}
and
\begin{align*}
\langle X^2, Y^2\rangle=\langle X^2Y, Y\rangle=\langle L_{X^2}(Y), Y\rangle.
\end{align*}
Fix $X$, then both sides are quadratic forms in $Y$.  Consider the spectral decomposition of $X$:
\begin{align*}
X=\sum_{i=1}^r \lambda_iH_i,
\end{align*}
where $\{H_1, \cdots, H_r\}$ is a Jordan frame and $\lambda_1, \cdots, \lambda_r \in \mathbb{R}$. According to the Peirce decomposition of $\mathscr{A}$ with respect to $\left\{H_1, \cdots, H_r\right\}$ (see Theorem~\ref{Peirce}), we have the following orthogonal direct sum
\begin{align*}
\mathscr{A}=\bigoplus_{i=1}^r\mathbb{R}H_i\bigoplus_{i<j} \mathscr{A}_{ij}.
\end{align*}
Let
\begin{align*}
Y=\sum_{i=1}^r \mu_iH_i+\sum_{i<j} Y_{ij},
\end{align*}
where $\mu_i \in \mathbb{R}$ and $Y_{ij} \in \mathscr{A}_{i j}$. Then
\begin{align*}
L_X(Y)=\sum_{i=1}^r\lambda_i\mu_i H_i+\sum_{i<j}\frac{\lambda_i+\lambda_j}{2} Y_{ij}.
\end{align*}
It follows that
\begin{align*}
\langle L_X^2(Y), Y\rangle& =\sum_{i=1}^r \lambda_i^2\mu_i^2+\sum_{i<j}(\frac{\lambda_i+\lambda_j}{2})^2\|Y_{ij}\|^2, \\
\langle L_{X^2}(Y), Y\rangle&=\sum_{i=1}^r \lambda_i^2\mu_i^2+\sum_{i<j}\frac{\lambda_i^2+\lambda_j^2}{2}\|Y_{ij}\|^2 .
\end{align*}
The inequality follows since
\begin{align*}
(\frac{\lambda_i+\lambda_j}{2})^2 \leq \frac{\lambda_i^2+\lambda_j^2}{2}.
\end{align*}
This completes the proof.
\end{proof}

\begin{exam}
Consider the  formally real Jordan algebra  $\operatorname{Sym}(n,\mathbb{R})$ in Remark~\ref{classificationlist}, i.e., the space of  $n\times n$  real symmetric matrices   with the usual Jordan multiplication   $A \circ B:=\frac{1}{2}(A B+B A),\forall A,B\in \operatorname{Sym}(n,\mathbb{R}).$  It is easily seen that
$$\langle A, B\rangle:=\operatorname{Tr} {AB^t}=\operatorname{Tr} {AB},~\forall A,B\in \operatorname{Sym}(n,\mathbb{R}),$$
is an associative inner product on $\operatorname{Sym}(n,\mathbb{R})$. In this case, Lemma~\ref{inequality} reads as follows
$$\operatorname{Tr}{(A\circ B)\circ(A\circ B)}\leq\operatorname{Tr}{(A\circ A)\circ(B\circ B)}, ~\forall A,B\in \operatorname{Sym}(n,\mathbb{R}),$$
which is also equivalent to the following known inequality
$$\operatorname{Tr}{ABAB}\leq\operatorname{Tr}{AABB},$$
for all $A,B\in \operatorname{Sym}(n,\mathbb{R}).$
\end{exam}

\begin{remark}\label{XY=}
By the proof of Lemma~\ref{inequality}, we know that if $\langle XY, XY\rangle=\langle X^2, Y^2\rangle$ for all $X, Y \in \mathscr{A}$, then the formally real Jordan algebra $\mathscr{A}$ is necessarily a direct sum of the one-dimensional simple Jordan algebra (i.e., $e_{1}e_1=e_1$, $r=1$).
\end{remark}

\begin{theorem}\label{nonpositive}
Let $\mathscr{A}$ be a formally real Jordan algebra, and $\langle\cdot, \cdot\rangle$ be an associative inner product on it. Then the Jordan curvature  of $(\mathscr{A},\langle\cdot, \cdot\rangle)$ is non-positive. Moreover, the Jordan curvature vanishes identically if and only $\mathscr{A}$ is a direct sum of the one-dimensional simple Jordan algebra.
\end{theorem}
\begin{proof}
By Lemma~\ref{1/2}, we know that
\begin{align*}
R(X,Y)X=\frac{1}{2}(XY)X-\frac{1}{4}(YX)X-\frac{1}{4}(X^2)Y=\frac{1}{4}(XY)X-\frac{1}{4}(X^2)Y,
\end{align*}
for any $X, Y \in \mathscr{A}$. Since the inner product $\langle\cdot, \cdot\rangle$ is associative, we have
\begin{align*}
\langle R(X, Y)X, Y\rangle=\frac{1}{4}\langle(XY)X-(X^2)Y, Y\rangle=\frac{1}{4}\langle XY, XY\rangle-\frac{1}{4}\langle X^2, Y^2\rangle.
\end{align*}
The theorem is completed by Lemma~\ref{inequality} and Remark~\ref{XY=}.
\end{proof}

\section{The Jordan Ricci curvature of metric Jordan algebras}\label{Ricci}
In this section, we explore Jordan-Einstein metrics on Jordan algebras. We show that every simple formally
real Jordan algebra (of  dimension at least  two) admits a Jordan-Einstein metric  of negative Jordan scalar curvature. For  nontrivial nilpotent Jordan algebras, we prove that they admit no Jordan-Einstein metrics. Moreover,  we construct some  examples of Jordan-Einstein metrics.

\subsection{Jordan-Einstein metrics on formally real Jordan algebras}
Let $\mathscr{A}$ be a simple formally real Jordan algebra of rank $r$ and $E$ be the identity element. Then with respect to a Jordan frame $\{H_1, H_2, \cdots, H_r\},$ $\mathscr{A}$ has the following Peirce decomposition (see Theorem~\ref{Peirce})
\begin{align*}
\mathscr{A}=\bigoplus_{i\leq j}\mathscr{A}_{ij}=\bigoplus_{i=1}^r\mathbb{R}H_i\bigoplus_{i<j} \mathscr{A}_{ij}.
\end{align*}
For $r \geq 3$, it is known that the number $d=\operatorname{dim} \mathscr{A}_{ij}$ does not depend on $i,j$. So
\begin{align}\label{dim}
n=\dim\mathscr{A}=\dim\bigoplus_{i\leq j}\mathscr{A}_{ij}=\sum_{i=1}^r \dim \mathscr{A}_{ii}+\sum_{1\leq i<j\leq r}\dim\mathscr{A}_{ij}=r+\frac{r(r-1)}{2}d.
\end{align}
For $r=2$, we set $d=n-2$. It follows from the trace formula in \cite{FK1994} that
\begin{align}\label{KilTr}
B(X,Y)=\operatorname{Tr} L_XL_Y=\alpha\operatorname{tr}(XY)+\frac{d}{4}\operatorname{tr}(X)\cdot\operatorname{tr}(Y), ~~\forall X, Y \in \mathscr{A},
\end{align}
where $\alpha=1+\frac{(r-2)}{4}d$.

Now, we state the main result of this section.
\begin{theorem}\label{JEM}
Every simple formally real Jordan algebra admits a Jordan-Einstein metric.
\end{theorem}
In order to prove Theorem~\ref{JEM}, we need some preparation work. For a simple formally real Jordan algebra $\mathscr{A}$, by Lemma~\ref{trass} and Proposition~\ref{FRJAC}, we know that $\operatorname{tr}(XY)$ a positive definite symmetric bilinear form which is associative. Using (\ref{ratio}), we have the following  result.

\begin{Lemma}\label{Ric0}
Let $(\mathscr{A},\langle\cdot,\cdot\rangle_0)$ be a simple formally real algebra of fixed dimension $n$ and rank $r$, where
\begin{align*}
\langle X, Y\rangle_0:=\operatorname{tr}(XY)=\frac{r}{n}\tau(X, Y)=\frac{r}{n} \operatorname{Tr}L_{XY}, ~~\forall X, Y \in \mathscr{A}.
\end{align*}
Then the mean curvature vector of $(\mathscr{A},\langle\cdot, \cdot\rangle_0)$ is
\begin{align*}
H_0=\frac{n}{r}E=\left(1+\frac{(r-1)}{2}d\right)E.
\end{align*}
Moreover, the Jordan Ricci tensor $\textnormal{Ric}_0=\mathrm{ric}_0$ of $(\mathscr{A},\langle\cdot,\cdot\rangle_0)$ is given by
\begin{align*}
\textnormal{Ric}_0(X, Y)=-\frac{rd}{16}\langle X, Y\rangle_0+\frac{d}{16}\textnormal{tr}(X)\cdot\textnormal{tr}(Y),
\end{align*}
for all $X, Y \in \mathscr{A}$.
\end{Lemma}
\begin{proof}
Let $\{E_i\}$ be an orthonormal basis of $(\mathscr{A},\langle\cdot,\cdot\rangle_0)$, then by Definition~\ref{MCV}, the mean curvature vector of $(\mathscr{A},\langle\cdot,\cdot\rangle_0)$ is
\begin{align*}
H_0=\sum_i E_i^2
\end{align*}
Since $\langle\cdot,\cdot\rangle_0$ is associative,  we have
\begin{align*}
\langle X, H_0\rangle_0=\sum_i\langle X, E_i^^2\rangle_0=\sum_i\langle XE_i, E_i\rangle_0=\operatorname{Tr} L_X=\frac{n}{r}\langle X, E\rangle_0,~~\forall X \in \mathscr{A}.
\end{align*}
In particular, $H_0=\frac{n}{r}E$. By (\ref{dim}), $n=r+\frac{r(r-1)}{2}d$, so we have $H_0=\left(1+\frac{(r-1)}{2}d\right)E$. This proves the first statement.

For the second statement. By Remark~\ref{ric=Ric}, we have $\textnormal{Ric}_0= \textnormal{ric}_0$. It follows from Theorem~\ref{RF} that the Jordan Ricci tensor is given by
\begin{align*}
\textnormal{Ric}_0(X, Y)=\langle\textnormal{M}_0 X, Y\rangle_0+\frac{1}{2} B(X, Y)-\frac{1}{4}\langle H_0, XY\rangle_0, ~~\forall X, Y \in \mathscr{A},
\end{align*}
where
\begin{align*}
\langle\textnormal{M}_0 X, Y\rangle_0& =-\frac{1}{2} \sum_{i,j}\langle XE_i, E_j\rangle_0\langle YE_i, E_j\rangle_0+\frac{1}{4} \sum_{i, j}\langle E_iE_j, X\rangle_0\langle E_iE_j, Y\rangle_0 \\
&=-\frac{1}{2} \operatorname{Tr} {L_XL_Y}+\frac{1}{4}\operatorname{Tr} {L_XL_Y} \\
&=-\frac{1}{4} B(X,Y).
\end{align*}
Since $H_0=\frac{n}{r}E=\left(1+\frac{(r-1)}{2}d\right)E$, then
\begin{align*}
\langle H_0, X Y\rangle_0=\frac{n}{r}\langle E, XY\rangle_0=\frac{n}{r}\langle X, Y\rangle_0=\left(1+\frac{(r-1)}{2} d\right)\langle X, Y\rangle_0.
\end{align*}
So
\begin{align*}
\textnormal{Ric}_0(X, Y)=\frac{1}{4} B(X, Y)-\frac{1}{4}\left(1+\frac{(r-1)}{2} d\right)\langle X, Y\rangle_0.
\end{align*}
On the other hand, by (\ref{KilTr}) we have
\begin{align*}
B(X, Y)=\operatorname{Tr} {L_XL_Y}=\left(1+\frac{(r-2)}{4}d\right) \operatorname{tr}(XY)+\frac{d}{4} \operatorname{tr}(X)\cdot\operatorname{tr}(Y).
\end{align*}
It follows that
\begin{align*}
\textnormal{Ric}_0(X, Y)=-\frac{rd}{16}\langle X, Y\rangle_0+\frac{d}{16} \operatorname{tr}(X)\cdot\operatorname{tr}(Y).
\end{align*}
This completes the proof of the Lemma.
\end{proof}

\begin{remark}
It is not hard to see that $\textnormal{Ric}_0 \leq 0$ and $\textnormal{Ric}_0(E, X)=\textnormal{Ric}(X, E)=0$ for any $X \in \mathscr{A}$.
\end{remark}

In the sequel, we show that it is possible to deform the inner product $\langle\cdot,\cdot\rangle_0$ in the direction $E$ so that it becomes a Jordan-Einstein metric on $\mathscr{A}$.

\textbf{Step one:} Consider the following basis of $\left(\mathscr{A},(\cdot, \cdot\rangle_0\right)$
\begin{align}\label{E-basis}
E_1=H_1,E_2=H_2,\cdots, E_r=H_r, \underbrace{E_{12}^{1},\cdots, E_{12}^{d}}_{\mathscr{A}_{12}},\cdots,\underbrace{E_{1r}^{1},\cdots, E_{1r}^{d}}_{\mathscr{A}_{1r}},\cdots, \underbrace{E_{(r-1)r}^{1},\cdots, E_{(r-1)r}^{d}}_{\mathscr{A}_{(r-1)r}}.
\end{align}
where $\{H_1, H_2, \cdots, H_r\}$ is a Jordan frame of $\mathscr{A}$, and $\{E_{ij}^1, \cdots, E_{ij}^d\}$ is an orthonormal basis of $(\mathscr{A}_{ij}, \langle\cdot, \cdot\rangle_0)$ for  $1\leq i<j\leq r$. By Theorem~\ref{Peirce}, we know that (\ref{E-basis}) is an orthonormal basis of $(\mathscr{A},\langle\cdot, \cdot\rangle_0)$. With respect to this basis, it is easy to see the Jordan Ricci tensor is represented by
\begin{align*}
\textnormal{Ric}_0=-\frac{rd}{16}I+\frac{d}{16}\left(\begin{array}{cccc|c}
1&1&\cdots & 1 & \textbf{0} \\
1&1&\cdots & 1 & \textbf{0} \\
\vdots & \vdots & \ddots & \vdots & \vdots \\
1 & 1 & \cdots & 1 & \textbf{0} \\
\hline
\textbf{0}&\textbf{0}&\cdots&\textbf{0}&\textbf{0}
\end{array}\right),
\end{align*}
where the square matrix $(\textbf{1})$ is of order $r.$

\textbf{Step two:} Let $\{F_1, F_2, \cdots, F_r\}$ be an orthonormal basis of $\textnormal{span}\{E_1, E_2, \cdots, E_r\}$, such that $F_1$ is a positive constant multiple of the identity element $E$. Then it follows that
\begin{align}\label{F-basis}
F_1,F_2,\cdots, F_r, \underbrace{E_{12}^{1},\cdots, E_{12}^{d}}_{\mathscr{A}_{12}},\cdots,\underbrace{E_{1r}^{1},\cdots, E_{1r}^{d}}_{\mathscr{A}_{1r}},\cdots, \underbrace{E_{(r-1)r}^{1},\cdots, E_{(r-1)r}^{d}}_{\mathscr{A}_{(r-1)r}},
\end{align}
is an orthonormal basis of $(\mathscr{A},\langle\cdot, \cdot\rangle_0)$. Since $\langle E, E\rangle_0=\textnormal{tr}(E)=r$, we conclude that
$F_1=\frac{1}{\sqrt{r}}E$. Moreover, $F_1, F_2, \cdots, F_r$ being orthogonal implies
\begin{align*}
\operatorname{tr}(F_2)=\cdots=\operatorname{tr}(F_r)=0.
\end{align*}
Hence relative to the basis (\ref{F-basis}),  the Jordan Ricci operator is represented by
\begin{align*}
\textnormal{Ric}_0=-\frac{r d}{16} I+\frac{d}{16}
\left(\begin{array}{cccc|c}
r & 0 & \cdots & 0 & \textbf{0} \\
0 & 0 & \cdots & 0 & \textbf{0}\\
\vdots & \vdots & \ddots & \vdots & \vdots \\
0 & 0 & \cdots & 0 & \textbf{0} \\
\hline \textbf{0} & \textbf{0} & \cdots & \textbf{0} & \textbf{0}
\end{array}\right).
\end{align*}

\textbf{Step three:} With respect to the basis (\ref{F-basis}), we define  a one-parameter subgroup $g_t \in \textnormal{GL}(n)$
\begin{align*}
g_t:=\left(\begin{array}{cccc}
e^{-t} & & & \\
& 1 & & \\
& & \ddots & \\
& & & 1
\end{array}\right), ~~t \in \mathbb{R},
\end{align*}
and  a family  of metrics
\begin{align}\label{t-metric}
\langle\cdot,\cdot\rangle_t:=\langle g_t^{-1}(\cdot), g_t^{-1}(\cdot)\rangle_0, ~~t \in \mathbb{R}.
\end{align}
Then it is easy to see that
\begin{align}\label{G-basis}
G_1=e^{-t}F_1,G_2=F_2,\cdots, G_r=F_r, \underbrace{E_{12}^{1},\cdots, E_{12}^{d}}_{\mathscr{A}_{12}},\cdots,\underbrace{E_{1r}^{1},\cdots, E_{1r}^{d}}_{\mathscr{A}_{1r}},\cdots, \underbrace{E_{(r-1)r}^{1},\cdots, E_{(r-1)r}^{d}}_{\mathscr{A}_{(r-1)r}},
\end{align}
is an orthonormal basis of $(\mathscr{A},\langle\cdot, \cdot\rangle_t)$ for any $t \in \mathbb{R}$.

\begin{Lemma}\label{Rict}
 Let the notations be as above. For  metric Jordan algebra $(\mathscr{A},\langle\cdot, \cdot\rangle_t)$, the mean curvature vector is given by
\begin{align*}
H_t=\frac{e^{-2 t}+n-1}{r}E=\frac{e^{-2t}+n-1}{\sqrt{r}}F_1,
\end{align*}
Moreover, with respect to the orthonormal basis \textnormal{(\ref{G-basis})}, we have
\begin{align*}
\textnormal{M}_t=\frac{1}{4}\left(\begin{array}{cccc}
-\frac{2n-1}{r}e^{-2t}+\frac{n-1}{r} e^{2t} & & & \\
& -\frac{2}{r}e^{2t}+c & & \\
& & \ddots&  \\
& & &-\frac{2}{r}e^{2t}+c
\end{array}\right),
\end{align*}
where $c=\frac{2}{r}-1-\frac{(r-2)}{4}d$.
\begin{align*}
\frac{1}{2}B=\left(\begin{array}{cccc}
\frac{1}{2}\left(1+\frac{(r-1)}{2} d\right) e^{-2t} & & & \\
& \frac{1}{2}\left(1+\frac{(r-2)}{4}d\right) & & \\
& & \ddots & \\
& & & \frac{1}{2}\left(1+\frac{(r-2)}{4}d\right)
\end{array}\right),
\end{align*}
and
\begin{align*}
-\frac{1}{4}(\langle H_t,\cdot\rangle_t)=
\left(\begin{array}{cccc}
-\frac{e^{-2t}+n-1}{4r} & & & \\
& -\frac{e^{-2t}+n-1}{4r} e^{2t} & & \\
& & \ddots & \\
& & & -\frac{e^{-2t}+n-1}{4r} e^{2t}
\end{array}\right).
\end{align*}
\end{Lemma}
\begin{proof}
Proof. Since (\ref{G-basis}) is an orthonormal basis of $(\mathscr{A},\langle\cdot, \cdot\rangle_t)$, then
\begin{align*}
H_t& =\sum_{i=1}^r G_i^2+\sum_{i<j}\sum_{k=1}^d E_{ij}^k \\
&=G_1^2+\sum_{j=2}^rG_j^2+\sum_{i<j}\sum_{k=1}^d E_{ij}^k \\
&=e^{-2t}F_1^2+\sum_{j=2}^r F_i^2+\sum_{i<j}\sum_{k=1}^d E_{ij}^k \\
&=e^{-2t}F_1^2+(H_0-F_1^2) \\
&=e^{-2t}\frac{1}{r}E+\left(\frac{n}{r} E-\frac{1}{r}E\right) \\
&=\frac{e^{-2t}+n-1}{r}E \\
&=\frac{e^{-2t}+n-1}{\sqrt{r}}F_1.
\end{align*}
It follows that
\begin{align*}
\langle H_t, XY\rangle_t=\langle\frac{e^{-2t}+n-1}{\sqrt{r}}F_1, XY\rangle_t=e^{2t}\langle\frac{e^{-2 t}+n-1}{\sqrt{r}} F_1, X Y\rangle_0=\frac{e^{-2t}+n-1}{r} e^{2t}\langle E, XY\rangle_0=\frac{e^{-2t}+n-1}{r}e^{2t}\langle X, Y\rangle_0,
\end{align*}
for any $X, Y \in \mathscr{A}$. In particular,
\begin{align*}
\langle H_t, G_1G_1\rangle_t& =\frac{e^{-2t}+n-1}{r}e^{2t}\langle G_1, G_1\rangle_0=\frac{e^{-2 t}+n-1}{r}. \\
\langle H_t, G_jG_j\rangle_t& =\frac{e^{-2t}+n-1}{r}e^{2t}\langle F_j, F_j\rangle_0=\frac{e^{-2t}+n-1}{r}e^{2t}, ~j \geq 2. \\
\langle H_t, G_iG_j\rangle_t& =0, ~i \neq j. \\
\langle H_t, E_{ij}^p E_{kl}^q\rangle_t&=\delta_{pq}\delta_{ik}\delta_{jl}\frac{e^{-2t}+n-1}{r}e^{2t}.
\end{align*}
Moreover
\begin{align*}
B(G_1, G_1)&=B(e^{-t}F_1, e^{-t}F_1)=\left(1+\frac{(r-1)}{2}d\right)e^{-2t}. \\
B(G_j, G_j)&=B(F_j, F_j)=1+\frac{(r-2)}{4}d, ~j \geq 2. \\
B(G_i, G_j)&=B(F_i, F_j)=0, ~i \neq j. \\
B(E_{ij}^p, E_{kl}^q) & =\delta_{pq}\delta_{ik}\delta_{jl}\left(1+\frac{(r-2)}{4}d\right).
\end{align*}
For the calculation of $\textnormal{M}_t$, it follows from a similar discussion as in \cite{GKM}. Note that $\textnormal{M}_t$ differs from the moment map in $\textnormal{M}$ in (\ref{Mf}) by a constant multiple of $\frac{1}{4}$. This completes the proof of the Lemma.
\end{proof}

Now, we are in a position to prove Theorem~\ref{JEM}.
\begin{proof} [\bf The proof of Theorem~\ref{JEM}] By Lemma~\ref{Rict}, we know that the Jordan Ricci tensor $\operatorname{Ric}_t$ of $(\mathscr{A},\langle\cdot,\rangle_t)$ is given by
\begin{align*}
\textnormal{Ric}_t(X, Y)=\langle\textnormal{M}_tX, Y\rangle_t+\frac{1}{2}B(X, Y)-\frac{1}{4}\langle H_t, XY\rangle_t, ~~\forall X, Y \in \mathscr{A}.
\end{align*}
To prove the theorem, it suffices to show that there exists some $t \in \mathbb{R}$ such that the following two functions
\begin{align}\label{LSH}
\frac{1}{4}\left(-\frac{2n-1}{r} e^{-2t}+\frac{n-1}{r} e^{2t}\right)+\frac{1}{2}\left(1+\frac{(r-1)}{2}d\right)e^{-2t}-\frac{e^{-2t}+n-1}{4r},
\end{align}
and
\begin{align}\label{RSH}
\frac{1}{4}\left(-\frac{2}{r}e^{2t}+\frac{2}{r}-1-\frac{(r-2)}{4}d\right)+\frac{1}{2}\left(1+\frac{(r-2)}{4}d\right)-\frac{e^{-2 t}+n-1}{4r} e^{2t}.
\end{align}
have the same value. The term of (\ref{LSH}) equals to
\begin{align*}
\left(-\frac{n}{2r}+\frac{1}{2}+\frac{(r-1)}{4}d\right)e^{-2 t}+\frac{(n-1)}{4r}e^{2t}-\frac{(n-1)}{4r}=\frac{(n-1)}{4r}e^{2 t}-\frac{(n-1)}{4r},
\end{align*}
and  the term of (\ref{RSH}) equals to
\begin{align*}
-\frac{(n+1)}{4r}e^{2t}+\frac{1}{4r}+\frac{1}{4}+\frac{(r-2)}{16}d.
\end{align*}
The solution of
\begin{align*}
\frac{(n-1)}{4r}e^{2t}-\frac{(n-1)}{4r}=-\frac{(n+1)}{4r} e^{2 t}+\frac{1}{4r}+\frac{1}{4}+\frac{(r-2)}{16}d
\end{align*}
is
\begin{align*}
t &=\frac{1}{2} \ln \left(\frac{1}{2}+\frac{r}{2n}\left(1+\frac{(r-2)}{4} d\right)\right) \\
& =\frac{1}{2} \ln \left(1-\frac{r d}{8+4(r-1) d}\right) \\
& <0,
\end{align*}
for $r\geq 2$ by Remark~\ref{classification}.
This completes the theorem.
\end{proof}

We exhibit  a specific example as follows
\begin{exam}
Consider the  formally real Jordan algebra  $\operatorname{Sym}(2,\mathbb{R})$, i.e., the space of  $2\times 2$  real symmetric matrices   with the usual Jordan multiplication. It is easy  to see that
\begin{align*}
e_1=\left(\begin{array}{cc}
1 &0 \\
0&0
\end{array}\right),\quad
e_2=\left(\begin{array}{cc}
0 &0 \\
0&1
\end{array}\right),\quad
e_3=\left(\begin{array}{cc}
0 &1 \\
1&0
\end{array}\right),
\end{align*}
is a basis of $\operatorname{Sym}(2,\mathbb{R})$. The Jordan multiplication table is given as follows
\begin{align*}
\left(\begin{array}{ccc}
e_1e_1 &e_1e_2 &e_1e_3  \\
e_2e_1 &e_2e_2 &e_2e_3  \\
e_3e_1 &e_3e_2 &e_3e_3
\end{array}\right)=
\left(\begin{array}{ccc}
e_1 &0 &\frac{1}{2}e_3  \\
0 &e_2 &\frac{1}{2}e_3  \\
\frac{1}{2}e_3  &\frac{1}{2}e_3  &e_1+e_2
\end{array}\right)
\end{align*}
The Jordan algebra  $\operatorname{Sym}(2,\mathbb{R})$ has $d=1$,  rank $2$ and  $\{e_1,e_2\}$ as a Jordan frame.   Consider the associative inner product
\begin{align*}
\langle X, Y\rangle_0:=\operatorname{tr}(XY)=\frac{2}{3} \operatorname{Tr}L_{XY}, ~~\forall X, Y\in\operatorname{Sym}(2,\mathbb{R}).
\end{align*}
From step one above, we get an  orthonormal basis of $(\operatorname{Sym}(2,\mathbb{R}),\langle \cdot, \cdot\rangle_0)$, i.e.,
\begin{align*}
E_1=e_1,\quad E_2=e_2,\quad E_{12}=\frac{1}{\sqrt{2}}e_3.
\end{align*}
From step two above, we get another orthonormal basis of $(\operatorname{Sym}(2,\mathbb{R}),\langle \cdot, \cdot\rangle_0)$, i.e.,
\begin{align*}
F_1=\frac{1}{\sqrt{2}}E,\quad F_2=\frac{1}{\sqrt{2}}E_1-\frac{1}{\sqrt{2}}E_2,\quad E_{12}.
\end{align*}
From step three above, we know that
\begin{align*}
G_1=e^{-t}F_1,\quad G_2=F_2,\quad E_{12},
\end{align*}
is an orthonormal basis of $(\operatorname{Sym}(2,\mathbb{R}),\langle \cdot, \cdot\rangle_t).$
The Jordan multiplication table is
\begin{align*}
\left(\begin{array}{ccc}
G_1G_1 &G_1G_2 &G_1E_{12}  \\
G_2G_1 &G_2G_2 &G_2E_{12}  \\
E_{12}G_1 &E_{12}G_2 &E_{12}E_{12}
\end{array}\right)=
\left(\begin{array}{ccc}
\frac{1}{\sqrt{2}}e^{-t}G_1 &\frac{1}{\sqrt{2}}e^{-t}G_2 &\frac{1}{\sqrt{2}}e^{-t}E_{12} \\
\frac{1}{\sqrt{2}}e^{-t}G_2&\frac{1}{\sqrt{2}}e^{t}G_1&0  \\
\frac{1}{\sqrt{2}}e^{-t}E_{12} &0 &\frac{1}{\sqrt{2}}e^{t}G_1
\end{array}\right).
\end{align*}
With respect to the orthonormal basis $\{G_1, G_2, E_{12}\}$, we have
\begin{align*}
\textnormal{M}_t=\frac{1}{4}\left(\begin{array}{cccc}
-\frac{5}{2}e^{-2t}+ e^{2t} &  & \\
& -e^{2t}  & \\
& & -e^{2t}
\end{array}\right),
\end{align*}
and
\begin{align*}
\frac{1}{2}B=\left(\begin{array}{cccc}
\frac{3}{4}e^{-2t} &  & \\
&\frac{1}{2}& \\
& &  \frac{1}{2}
\end{array}\right).
\end{align*}
Since $H_t=G_1^2+G_2^2+E_{12}^2=\frac{e^{-t}+2e^{t}}{\sqrt{2}}G_1$, then
\begin{align*}
-\frac{1}{4}(\langle H_t,\cdot\rangle_t)=
\left(\begin{array}{cccc}
-\frac{e^{-2t}+2}{8} &  & \\
& -\frac{e^{-2t}+2}{8} e^{2t}  & \\
&  & -\frac{e^{-2t}+2}{8} e^{2t}
\end{array}\right).
\end{align*}
It follows that $(\mathscr{A},\langle\cdot,\rangle_t)$ is Jordan-Einstein if and only if
\begin{align*}
\frac{1}{4}\left(-\frac{5}{2} e^{-2t}+ e^{2t}\right)+\frac{3}{4}e^{-2t}-\frac{e^{-2t}+2}{8}=-\frac{1}{4}e^{2t}+\frac{1}{2}-\frac{e^{-2t}+2}{8}e^{2t}.
\end{align*}
It is a straightforward calculation  to see $t=\frac{1}{2}\ln(\frac{5}{6}).$ So the  formally real Jordan algebra  $\operatorname{Sym}(2,\mathbb{R})$ admits a Jordan-Einstein metric.
\end{exam}

\begin{remark}
Let $\mathscr{A}$ be a simple formally real Jordan algebra of rank $\geq 2$. Then the Jordan-Einstein metric constructed in the proof of Theorem~\ref{JEM} has negative scalar curvature.
\end{remark}

\begin{Corollary}\label{Corofss}
Formally real Jordan algebras  containing no one-dimensional simple ideal admit a Jordan-Einstein metric of negative scalar curvature.
\end{Corollary}

\begin{exam}\label{Ric0+}
In the two dimension case, there are precisely two real semisimple Jordan algebras, i.e.,
$$
\begin{aligned}
& \psi_1: e_1 e_1=e_1, ~e_2 e_2=e_2. \\
& \psi_2: e_1 e_1=e_1, ~e_2 e_2=-e_1, ~e_1 e_2=e_2.
\end{aligned}
$$
The semisimple Jordan algebras $\psi_1$ and $\psi_2$ have the same complexification. Moreover, $\psi_1$ is formally real, and  $\psi_2$ is simple. Endow $\psi_1$ and $\psi_2$ with the metric $\langle\cdot, \cdot\rangle$ so that $\left\{e_1, e_2\right\}$ is an orthonormal basis. Then by Theorem~\ref{RF} and a straightforward calculation, we have
\begin{align*}
\textnormal{Ric}_{\psi_1}=0, \quad \textnormal{Ric}_{\psi_2}=\frac{1}{4}\langle\cdot, \cdot\rangle.
\end{align*}
That is, $\left(\psi_1,\langle\cdot, \cdot\rangle\right)$ and $\left(\psi_2,\langle\cdot, \cdot\rangle\right)$ are both Jordan-Einstein.
\end{exam}

\subsection{Nilpotent Jordan algebras admit no Jordan-Einstein metrics }
Let $(\mathscr{N},\langle\cdot, \cdot\rangle)$ be a nilpotent metric Jordan algebra. Since the Killing form of $\mathscr{N}$ necessarily vanishes, then by Theorem~\ref{RF} we have
\begin{align*}
\textnormal{Ric}(X, Y)=\langle\textnormal{M}X, Y\rangle-\frac{1}{4}\langle H, XY\rangle, ~~\forall X, Y \in \mathscr{N},
\end{align*}
where $H=\sum_{i=1}^n E_i^2$ is the mean curvature vector, $\{E_i\}$ is an orthonormal basis of $(\mathscr{N},\langle\cdot, \cdot\rangle)$ and
\begin{align*}
\langle\textnormal{M}X, Y\rangle=-\frac{1}{2}\sum_{i,j}\langle XE_i, E_j\rangle\langle YE_i, E_j\rangle+\frac{1}{4} \sum_{i,j}\langle E_iE_j, X\rangle\langle E_iE_j, Y\rangle, ~~\forall X, Y \in \mathscr{N} .
\end{align*}
It is clear that we  have
\begin{align*}
\operatorname{Tr} \textnormal{M}=-\frac{1}{4}\sum_{i,j,k}\langle E_iE_j, E_k\rangle^2 \leq 0,
\end{align*}
and the equality holds if and only $\mathscr{N}$ is the trivial algebra. Moreover, if $0\ne Z$ lies in the annihilator of $\mathscr{N}$, then
\begin{align*}
\langle\textnormal{M}Z, Z\rangle=\frac{1}{4}\sum_{i,j}\langle E_iE_j, Z\rangle\langle E_i E_j, Z\rangle \geq 0.
\end{align*}

\begin{theorem}\label{nilpotent}
A nontrivial nilpotent Jordan algebra admits no Jordan-Einstein metrics.
\end{theorem}
\begin{proof}
Suppose to the contrary that $(\mathscr{N},\langle\cdot, \cdot\rangle)$ is a nontrivial nilpotent metric Jordan algebra with
\begin{align*}
\textnormal{Ric}=c\langle\cdot, \cdot\rangle,
\end{align*}
for some constant $c \in \mathbb{R}$. Since the Killing form $B$ of $\mathscr{N}$ vanishes, then
\begin{align*}
c\cdot\operatorname{dim}\mathscr{N}=\operatorname{Tr}\textnormal{Ric}=\operatorname{Tr} \textnormal{M}+\frac{1}{2} \operatorname{Tr}B
-\frac{1}{4}\langle H, H\rangle=\operatorname{Tr} \textnormal{M}-\frac{1}{4}\langle H, H\rangle<0.
\end{align*}
Consequently, $c<0$. On the the hand, let $0 \neq Z$ be an element lying in the annihilator of $\mathscr{N}$. Then
\begin{align*}
c\langle Z, Z\rangle=\textnormal{Ric}(Z, Z)=\langle\textnormal{M} Z, Z\rangle-\frac{1}{4}\langle H, Z^2\rangle=\langle\textnormal{M}Z, Z\rangle \geq 0.
\end{align*}
It follows that $c\geq0$, which is contradiction. So a nontrivial nilpotent Jordan algebra admits no Jordan-Einstein metrics.
This proves the theorem.
\end{proof}

\begin{exam}\label{semidirect}
Consider the $n$-dimensional Jordan algebra $\mathscr{A}$
\begin{align*}
e_1 e_1=e_1,~ e_1 e_2=e_2, ~\cdots, ~e_1 e_n=e_n.
\end{align*}
It is neither nilpotent nor semisimple.
Endow $\mathscr{A}$ with a metric $\langle\cdot, \cdot\rangle$ so that $\{e_1, e_2, \cdots, e_n\}$ is an orthonormal basis. Then by a straightforward calculation, we have
\begin{align*}
\textnormal{Ric}=0.
\end{align*}
That is, $\langle\cdot, \cdot\rangle$ is a flat Jordan-Einstein metric on $\mathscr{A}$.
\end{exam}

\section{Further study}\label{Summary}
For a metric Lie algebra $(\g,\langle\cdot, \cdot\rangle)$, there is an important notion called \textit{algebraic Ricci solitons}, i.e., the Ricci operator satisfies  $$\textnormal{Ric}=cI+D,~~~~\textnormal{for some}~ c\in \mathbb{R},~D \in \operatorname{Der}(\mathfrak{g}).$$
The notion arises as limits under Ricci flow can be naturally seen as a generalization of Einstein metric (\cite{Jablonski2014,Jablonski2015H,lauret2001}).
It is well-known that nilpotent Lie algebras admit no Einstein metrics,  but  many of them admit \textit{a nilsoliton metric} (i.e., algebraic Ricci soliton in the nilpotent case). This is also true for nilpotent Jordan algebras (see Theorem~\ref{nilpotent} and \cite{GKM}). For nilpotent Lie algebras, nilsolitons can be characterized by a solvable  Einstein extension, and we don't know whether  or not it holds for nilpotent metric Jordan algebras. Besides,  \textit{classifying Jordan-Einstein metrics on formally real Jordan algebras up to isometry and scaling}, is  of interest,  noting that the Lie version is still  open (see \cite{BWZ04}).

Generally speaking, the differential geometry of invariant metrics on Lie (super-)algebras is a classical topic, having important applications to homogeneous spaces. We try to build a similar theory for metric Jordan algebras in this paper, hoping to  provide a  way to  understand homogeneous (super-)geometry by this Jordan theory.




\begin{appendix}

\section{Curvatures of metric Lie algebras}\label{Appendix-Lie}
Let $G$ be a connected Lie group with the Lie algebra $\g$ consisting
of left-invariant vector fields, and $\langle\cdot,\cdot\rangle$ be a left-invariant Riemannian metric on $G$. It is well-known
that we may identify $(G,\langle\cdot,\cdot\rangle)$ with the metric Lie algebra $(\g,\langle\cdot,\cdot\rangle)$ (see \cite{besse87book,Heber}). Let $\nabla$ be the Levi-Civita connection associated with $(\g,\langle\cdot,\cdot\rangle)$ and $X, Y, Z, U, V \in \g$. Then the torsion-free property and the
metric-preserving property of $\nabla$ are respectively given by
\begin{align}
[X,Y]&=\nabla_XY-\nabla_YX, \label{tor-f}\\
\langle \nabla_XY, Z&\rangle+\langle Y,\nabla_XZ\rangle=0.\label{metric-skew}
\end{align}
By Koszul's formula, one knows that the Levi-Civita connection $\nabla$ of $(\g,\langle\cdot,\cdot\rangle)$ is uniquely determined by
the equations (\ref{tor-f}) and (\ref{metric-skew}), that is,
\begin{align*}
\langle \nabla_XY,&Z\rangle=\frac{1}{2}(\langle [X,Y],Z\rangle)-\langle [Y,Z],X\rangle)+\langle [Z,X],Y\rangle).
\end{align*}
Associated with the Levi-Civita connection $\nabla$, the Riemann curvature tensor of $(\g,\langle\cdot,\cdot\rangle)$ is defined by
\begin{align*}
R(X,Y)Z=-\nabla_X\nabla_YZ+\nabla_Y\nabla_XZ+\nabla_{[X,Y]}Z,
\end{align*}
or equivalently, $R(X,Y)=\nabla_{[X,Y]}-[\nabla_X,\nabla_Y]$. We remark that a frequently encountered definition of
Riemann curvature tensor in the literature differs from the above by a sign! The $(0,2)$-type Ricci tensor
\textnormal{Ric} of $(\g,\langle\cdot,\cdot\rangle)$  is a trace or contraction of $R$, i.e.,
\begin{align*}
\textnormal{Ric}=\operatorname{Tr} {(X\rightarrow R(U,X)V)}.
\end{align*}
By the fundamental symmetry properties of $R$, one knows that \textnormal{Ric} is a symmetric tensor. Furthermore, by taking the
trace of \textnormal{Ric}, we obtain the scalar curvature \textnormal{sc} of $(\g,\langle\cdot,\cdot\rangle)$, i.e.,
\begin{align*}
\textnormal{sc}=\operatorname{Tr}_{\langle\cdot,\cdot\rangle}\textnormal{Ric}.
\end{align*}
In the frame of  metric Lie algebra $(\g,\langle\cdot,\cdot\rangle)$, there is a \textit{mean curvature vector} $H\in\g$, which is defined by
\begin{align*}
\langle H, X\rangle=\operatorname{Tr}\operatorname{ad}X, ~\forall X\in\g.
\end{align*}
Note that the Lie algebra $\g$ is \textit{unimodular} if and only if $H=0$.

\begin{Lemma}\label{RicinOrth}
Let $\{E_i\}$ be an arbitrary orthonormal basis of $(\g,\langle\cdot,\cdot\rangle).$ Then
\begin{align*}
\textnormal{Ric}(X,Y)=-&\frac{1}{2}\sum_{i,j}\langle[X,E_i],E_j\rangle\langle[Y,E_i],E_j\rangle
+\frac{1}{4}\sum_{i,j}\langle[E_i,E_j],X\rangle\langle[E_i,E_j],Y\rangle\\
&-\frac{1}{2}B(X,Y)-\frac{1}{2}(\langle [H,X], Y\rangle+\langle X, [H,Y]\rangle), ~~\forall X,Y\in\g,
\end{align*}
where $B$ is the Killing form of $\g,$ and $H$ is the mean curvature vector of $(\g,\langle\cdot,\cdot\rangle)$.
\end{Lemma}

The Ricci operator Ric (by abuse of notation) of $(\g,\langle\cdot,\cdot\rangle)$ is defined by $\langle\textnormal{Ric}(\cdot),(\cdot)\rangle=\textnormal{Ric}(\cdot,\cdot).$ It follows  from Lemma~\ref{RicinOrth} that  the Ricci operator is given by
\begin{align}\label{Riccif}
\textnormal{Ric}=\textnormal{M}-\frac{1}{2}B-S(\operatorname{ad}H),
\end{align}
where $B$ denotes the symmetric map defined by the Killing form  relative to $\langle\cdot,\cdot\rangle,$
$S(\operatorname{ad}H)=\frac{1}{2}(\operatorname{ad}H+(\operatorname{ad}H)^t)$ and \textnormal{M} is the symmetric map defined by
\begin{align}\label{Mformula}
\langle\textnormal{M}X,Y\rangle=-\frac{1}{2}\sum_{i,j}\langle[X,E_i],E_j\rangle\langle[Y,E_i],E_j\rangle
+\frac{1}{4}\sum_{i,j}\langle[E_i,E_j],X\rangle\langle[E_i,E_j]
,Y\rangle,~~\forall X,Y\in\g.
\end{align}
Note that for a nilpotent metric Lie algebra $(\mathfrak{n},\langle\cdot,\cdot\rangle)$, the Ricci operator of $(\mathfrak{n},\langle\cdot,\cdot\rangle)$ coincides with $\textnormal{M}.$

\section{Formally real Jordan algebras}\label{frja}
\begin{Lemma}[\cite{FK1994}]\label{Fundamenta}
Let $\mathscr{A}$ be an arbitrary  Jordan algebra.  Then the following identities hold:
\begin{enumerate}
\item [\textnormal{(i)}] $[L_X, L_{Y^2}]+2[L_Y, L_{XY}]=0$,
\item [\textnormal{(ii)}] $[L_X, L_{YZ}]+[L_Y, L_{ZX}]+[L_Z, L_{XY}]=0$,
\item [\textnormal{(iii)}]  $L_{X^2Y}-L_{X^2}L_Y=2(L_{XY}-L_XL_Y)L_X$,
\end{enumerate}
for any $X, Y, Z \in \mathscr{A}$.
\end{Lemma}

A Jordan algebra $\mathscr{A}$ over $\mathbb{R}$ with an identity element $E$ is called a \textit{formally real Jordan algebra} if $X^2+Y^2=0 \Rightarrow X=0, Y=0$. This in particular implies that the mean curvature vector of   metric formally real Jordan algebras never vanishes (see Definition~\ref{MCV}).

The following proposition characterizes formally real Jordan algebras.
\begin{Proposition}[\cite{FK1994}]\label{FRJAC}
 Let $\mathscr{A}$ be a Jordan algebra over $\mathbb{R}$ with an identity element. The following statements are equivalent
 \begin{enumerate}
\item [\textnormal{(i)}] $\mathscr{A}$ is a formally real Jordan algebra.
\item [\textnormal{(ii)}] There exists a positive definite symmetric bilinear form on $\mathscr{A}$ which is associative.
\item [\textnormal{(iii)}] The symmetric bilinear form $\operatorname{tr}(XY)$ is positive definite.
\item [\textnormal{(iv)}] The symmetric bilinear form $\tau(X, Y)=\operatorname{Tr} L_{XY}$ is positive definite.
\end{enumerate}
\end{Proposition}
 The proof of (ii) $\Rightarrow$ (i) in Proposition~\ref{FRJAC} follows from the fact that if $\langle\cdot,\cdot\rangle$ is an associative inner product on $\mathscr{A}$ and $X^2+Y^2=0$, then $\langle X^2+Y^2, E\rangle=\langle X, X\rangle+\langle Y, Y\rangle=0$, so $X=Y=0$.

\begin{theorem}[\cite{FK1994}]\label{FRJSS}
Every formally real Jordan algebra is semisimple, which decomposes, in a unique way, a direct sum of simple ideals. Besides, every complex semisimple Jordan algebra is the complexification of some formally real Jordan algebra.
\end{theorem}

Moreover
\begin{proposition}[\cite{FK1994}]\label{simpleuni}
In a simple formally real Jordan algebra, the associative symmetric bilinear form is unique up to a scalar. In particular, every associative symmetric bilinear form is a scalar multiple of $\operatorname{tr}(XY).$
\end{proposition}

\begin{remark}
Combining the results above, one may reasonably consider formally real Jordan algebras as the counterparts of compact real forms in complex semisimple Lie algebras.
\end{remark}

In the sequel, we recall the Jordan frame and the corresponding Peirce decomposition. Assume that $\mathscr{A}$ is a simple formally real Jordan algebra of fixed dimension $n$ and rank $r \geq 2$. By Lemma~\ref{tau}, Lemma~\ref{trass} and Proposition~\ref{simpleuni}, we have
\begin{align}\label{ratio}
\operatorname{tr}(XY)=\frac{r}{n} \tau(X, Y)=\frac{r}{n} \operatorname{Tr}L_{XY}, ~~\forall X, Y \in \mathscr{A}.
\end{align}
A non-zero element $C$ of $\mathscr{A}$ is called a \textit{primitive idempotent}, if $C^2=C$ and $C$ cannot be written as the sum of two non-zero idempotents. It is known that there exists a \textit{Jordan frame}, i.e., a set $\{H_1, H_2, \cdots, H_r\}$ of primitive idempotents of $\mathscr{A}$ such that $H_{1}+H_2+\cdots+H_r$ is the identity element $E$, and $H_iH_j=0$ for all $i \neq j$. Jordan frames in $\mathscr{A}$ are unique up to automorphisms. Moreover, for any $X \in \mathscr{A}$, we can always find a Jordan frame $\{H_1, H_2, \cdots, H_r\}$ such that
\begin{align*}
X=\lambda_1H_1+\cdots+\lambda_rH_r,
\end{align*}
where the numbers $\lambda_1, \cdots, \lambda_r \in \mathbb{R}$ are uniquely determined by $X$. In this case, we have $\operatorname{tr}(X)=\lambda_1+\cdots+\lambda_r$ (compare (\ref{tr})). In particular, $\operatorname{tr}(H_i)=1$ for all $1 \leq i \leq r$, and $\operatorname{tr}(E)=r$.

Now, let us fix a Jordan frame $\{H_1, H_2, \cdots, H_r\}$ of $\mathscr{A}$. By Proposition~\ref{Fundamenta} (i), we have $[L_{H_i}, L_{H_j}]=0$ for all $1 \leq i, j \leq r$. Using Proposition~\ref{Fundamenta} (iii) for $X=Y=H_i$, we obtain
\begin{align*}
2L_{H_i}^3-3 L_{H_i}^2+L_{H_i}=0.
\end{align*}
Therefore, an eigenvalue $\lambda$ of $L_{H_i}$ is a solution of
\begin{align*}
2\lambda^3-3 \lambda^2+\lambda=0,
\end{align*}
whose roots are $0, \frac{1}{2}$ and $1$. Noting that the operators $L_{H_i}$ are diagonalizable and commute with each other, they admit simultaneously diagonalization.
Consider the following subspaces of $\mathscr{A}$
\begin{align*}
\mathscr {A}_{ii}&=\mathscr{A}(H_i,1)=\mathbb{R}H_i, \\
\mathscr {A}_{ij}&=\mathscr{A}(H_i, \frac{1}{2}) \cap \mathscr{A}(H_j, \frac{1}{2}), ~~i < j,
\end{align*}
where $\mathscr{A}(H_i,1)$ and $\mathscr{A}(H_i, \frac{1}{2})$ denote the eigenspaces of $L_{H_i}$ corresponding to eigenvalues $1$ and $\frac{1}{2}$, respectively. Then we have the following famous Peirce decomposition of $\mathscr{A}$.

\begin{theorem}[\cite{FK1994}]\label{Peirce}
Let $\mathscr{A}$ be a simple formally real Jordan algebra of rank $r$. Then $\mathscr{A}$ decomposes in
the following orthogonal direct sum (with respect to an associative inner product, thus for all)
\begin{align*}
\mathscr{A}=\bigoplus_{i\leq j}\mathscr{A}_{ij}=\bigoplus_{i=1}^r\mathbb{R}H_i\bigoplus_{i<j} \mathscr{A}_{ij}.
\end{align*}
Moreover,
\begin{align*}
\mathscr{A}_{ij}\mathscr{A}_{ij} &\subset \mathscr{A}_{ii}+\mathscr{A}_{jj}, \\
\mathscr{A}_{ij}\mathscr{A}_{jk} &\subset \mathscr{A}_{ik}, \\
\mathscr{A}_{ij}\mathscr{A}_{kl}&=\{0\}, \text { if }\{i, j\} \cap\{k, I\}=\varnothing,
\end{align*}
where $1 \leq i, j, k, l \leq r$.
\end{theorem}
Note that $H_1+H_2+\cdots+H_r=E$, then we have $\operatorname{Tr} L_X=0=\operatorname{tr}(X)$ for any $X \in \oplus_{i<j}\mathscr{A}_{ij}$.

\begin{remark}\label{classification}
Every simple formally real Jordan algebra is isomorphic to one of the following list, where the number  $d:=\operatorname{dim} \mathscr{A}_{ij}$ does not depend on $i,j.$
\begin{table}[H]\label{classificationlist}
	\centering
	\begin{tabularx}{\textwidth}{|l|X|X|l|l|}
		\toprule
	 \quad\quad\quad\textbf{$\mathscr{A}$} & \quad\quad\quad\quad\textbf{Objects} &\quad\quad\textbf{Multiplication rule}&\textbf{Rank}\quad& \quad\textbf{$d$}  \\
		\midrule
\quad$\operatorname{Sym}(n,\mathbb{R})$  &$n\times n$ self-adjoint real matrices  &\quad\quad$A \circ B=\frac{1}{2}(A B+B A)$  &\quad$n$\quad &\quad$1$\\
       \midrule
\quad$\operatorname{Herm}(n,\mathbb{C})$  &$n\times n$ self-adjoint complex matrices  &\quad\quad$A \circ B=\frac{1}{2}(A B+B A)$  &\quad$n$\quad &\quad$2$\\
      \midrule
\quad$\operatorname{Herm}(n,\mathbb{H})$  &$n\times n$ self-adjoint  quatemionic matrices  &\quad\quad$A \circ B=\frac{1}{2}(A B+B A)$  &\quad$n$\quad &\quad$4$\\
     \midrule
\quad$\operatorname{Herm}(3,\mathbb{O})$  &$3\times 3$ self-adjoint   octonionic  matrices  &\quad\quad$A \circ B=\frac{1}{2}(A B+B A)$  &\quad$3$\quad &\quad$8$\\
     \midrule
The spin factors
& $\mathbb{R} \times \mathbb{R}^{n-1}$ and $f$ is a positive symmetric bilinear form on $\mathbb{R}^{n-1},~n\geq 3$  & $
(s, x)(t, y)
=(st+f(u, v), tx+s y)
$
 &\quad$2$ \quad&$n-2$\\
		\bottomrule
	\end{tabularx}
	\caption{Simple formally real Jordan algebra}
\end{table}%
\end{remark}
\end{appendix}

\section*{Acknowledgement}
We would like to thank Professor Jacques Faraut for providing us an elegant proof of Lemma~\ref{inequality}. This paper is partially supported by NSFC (Grant Nos. 11701300, 11931009 and 12131012), NSF of Jiangsu (Grant
No. BK20230803), the Fundamental Research Funds for the Central Universities (Grant No. 4007012303), and Guangdong Basic and Applied Basic Research Foundation (Grant No. 2023A1515010001)

\end{document}